\documentclass[12pt, a4paper]{amsart} 

\usepackage{mathptmx} 
\usepackage{amsfonts,amsthm,amssymb,dsfont,stmaryrd,mathrsfs} 
\usepackage[leqno]{amsmath} 
\usepackage{array}
\usepackage{colortbl}

\usepackage{graphicx,caption} 
\usepackage[pdfstartview=FitH]{hyperref} 
\usepackage{pdfpages} 
\usepackage{longtable,supertabular, multirow} 
\usepackage{doi}
\usepackage{tikz, tikz-cd}
\usetikzlibrary{matrix,arrows,decorations.pathmorphing}

\usepackage[hmarginratio={1:1}, vmarginratio={1:1}, top=3cm, bottom=2cm, left=3cm, right=3cm]{geometry} 

\usepackage{mathtools} 

\theoremstyle{plain}
\newtheorem*{thm*}{Theorem} 
\newtheorem{thm}{Theorem}[section] 
\newtheorem*{thm2*}{Theorem \ref{thmDecNilp}}
\newtheorem*{thm3*}{Theorem \ref{thmDecNilp2}}

\newtheorem{thmA}{Theorem}
\newtheorem{lem}[thm]{Lemma}
\newtheorem*{lem*}{Lemma}
\newtheorem{prop}[thm]{Proposition}
\newtheorem*{prop*}{Proposition}

\newtheorem{cor}[thm]{Corollary}
\newtheorem*{cor*}{Corollary}

\theoremstyle{definition}
\newtheorem{defn}[thm]{Definition}
\newtheorem*{defn*}{Definition}
\newtheorem{conjecture}{Conjecture}
\newtheorem*{conjecture*}{Conjecture}
\newtheorem*{conjecture3*}{Conjecture \ref{conjND}}
\newtheorem{exmp}[thm]{Example}
\newtheorem*{exmp*}{Example}

\newtheorem{prob}{Problem}
\newtheorem*{prob*}{Problem}

\newtheorem*{ques*}{Question}

\theoremstyle{remark}
\newtheorem{remk}[thm]{Remark}
\newtheorem*{remk*}{Remark}

\def\nn{\mathbb{N}}
\def\zz{\mathbb{Z}}
\def\qq{\mathbb{Q}}

\def\cc{\mathbb{C}}
\def\ff{\mathbb{F}}
\def\hh{{\mathbb H}}

\def\gcd{{\rm gcd}}

\def\supp{{\rm supp}}

\def\GL{{\rm GL}}

\def\Aut{{\rm Aut}}

\def\ord{{\rm ord}}

\def\Gal{{\rm Gal}}

\def\cen{{\rm Cen}}

\def\aug{{\rm aug}}
\def\ND{\textup{ND}} 
\def\SN{\textup{SN}}

\begin{document}

\title{Nilpotent Decomposition in Integral Group Rings}

\author{Eric Jespers}
\address{Department of Mathematics and Data Science, Vrije Universiteit Brussel, Pleinlaan 2, B-1050 Elsene, Belgium}
\email{eric.jespers@vub.be}

\author{Wei-Liang Sun}
\address{Department of Mathematics, National Taiwan Normal University, Taipei 11677, Taiwan, ROC}
\email{wlsun@ntnu.edu.tw}

\maketitle

\begin{abstract}
A finite group  $G$ is said to have the nilpotent decomposition property (ND)  if for every nilpotent element $\alpha $ of the integral group ring  $\zz[G]$ one has that $\alpha e$ also belong to $\zz[G]$, for every primitive central idempotent  $e$ of the rational group algebra $\qq[G]$. Results of Hales, Passi and Wilson, Liu and Passman show that this property is fundamental in the investigations of the multiplicative Jordan decomposition of integral group rings. If $G$ and all its subgroups have ND then Liu and Passman showed that $G$ has property SSN, that is, for subgroups $H$, $Y$ and $N$ of $G$, if $N\lhd H $ and $Y\subseteq H$ then $N\subseteq Y$ or $YN$ is normal in $H$; and such groups have been described.
In this article, we study the nilpotent decomposition property in integral group rings and we  classify finite SSN groups $G$ such that the rational group algebra $\qq[G]$ has only one Wedderburn component which is not a division ring.\\

Keywords: integral group ring, rational group algebra, nilpotent decomposition, multiplicative Jordan decomposition, Wedderburn component, Shoda pair, strong Shoda pair, SN group, SSN group.
\end{abstract}

\section{Introduction}
\label{Int}
Let $G$ be a finite group.
Because $\qq$ is a perfect field, every unit $\alpha \in \qq[G]$ has a unique multiplicative  Jordan decomposition,
that is $\alpha $ can be written uniquely as the product of a unit  $\alpha_s$ that is semisimple and  a commuting unipotent unit $a_u$. 
Following  \cite{HalPasWil}, one says  that the integral group ring $\zz[G]$ has the multiplicative Jordan decomposition property (MJD) (or simply, $G$ has (MJD) if for every unit of $\zz[G]$, its semisimple and unipotent parts are both contained in $\zz[G]$. 
By Maschke's Theorem, the semisimple algebra $\qq[G]=M_{n_1}(D_1) \times \cdots \times M_{n_m}(D_m)$, a direct product of matrices over skew fields $D_i$.
We call $n_1,\ldots , n_k$ the reduced degrees of $G$.
These matrix algebras we call the Wedderburn components of $\qq[G]$. 
Obviously, if $\qq[G]$ does not have nilpotent elements, i.e. when all $n_i=1$, then $G$ has (MJD). It is well-known that these groups are the abelian groups and the Hamiltonian groups of order $2^k t$ with $t$ an odd positive integer such that the multiplicative order of $2$ mod $t$ is odd. The (MJD) for $G$ has been intensively studied by Hales, Passi, Wilson, Liu and Passman and Kuo and Sun. In particular in  \cite{HalPasWil} it is shown that if $\qq[G]$ does have nonzero nilpotent elements and $G$ has (MJD) then $|G|=2^a3^b$ or $G=Q_8\times C_p$ or $G= C_p\rtimes C_n$ with $n=2^a$ or $n=3^b$ (here $C_n$ denotes a cyclic group of order $n$). Furthermore the reduced degrees of $G$ are at most $3$. Liu and Passman have  shown that the only non-abelian 3-groups with (MJD) are the groups of order 27. It remains an open problem to classify the finite groups that have (MJD).
 
In \cite{HalPasWil} another remarkable fact (and fundamental in the investigations) of groups $G$ with (MJD) has been shown: if $\alpha $ is a nilpotent element of $\zz[G]$ then all its components $\alpha e$ also belong to $\zz[G]$, where $e$ runs through the primitive central idempotents of $\qq[G]$. We call this property the  {\it nilpotent decomposition property} (ND for short) of $\zz[G]$ (or simply $G$). 
 In \cite[Proposition 2.5]{LiuPas}, C.-H. Liu and D.S. Passman  essentially proved that if $\zz[G]$ has ND, then $G$ is an SN group,
 that is for any subgroup $Y$ of $G$ and any normal subgroup $N$ of $G$  we have  $N\subseteq Y$ or $YN \lhd G$.
 A group is said to have SSN if all its subgroups have SN. Since the (MJD) property is inherited by subgroups, it follows that a finite group $G$ is an SSN group if $\zz[G]$ has (MJD).
 In \cite{LiuPas5}, Liu and Passman classified the finite SSN groups, and  in \cite{LiuPas,LiuPas2,LiuPas3,LiuPas4} they 
 classified the finite groups $G$  having (MJD) in case $G$ is a $3$-group or a  $\{2, 3\}$-group. Liu \cite{Liu} also gave a much simpler proof to determine $2$-groups $G$ such that $\zz[G]$ has (MJD).

In order to obtain a full classification of all finite groups $G$ with the (MJD) property it thus is crucial to study which finite SSN groups $G$ have the ND property.
This is the main aim of this paper. For a finite group $G$, we call a Wedderburn component of $\qq[G]$ a {\it matrix component} if it has reduced degree larger than $1$. Obviously, if $\qq[G]$ has at most one matrix component, then $G$ has ND. Finite groups $G$ such that $\qq[G]$  does not have a matrix component are
those for which $\qq[G]$ has no nilpotent elements and thus these are  well-known. 
It is natural to ask when $\qq[G]$ has precisely one matrix component. In this article, we classify finite SSN groups $G$ such that $\qq[G]$ has only one matrix component.
In  Theorem~\ref{thmDecNilp2}  we deal with the nilpotent groups and in Theorem~ \ref{thmDecNilp} we handle the groups that are not nilpotent. We also describe  the matrix component for each group in both theorems. In our investigations we did  not discover any finite group $G$ that has ND but has more than one matrix component and we actually expect that such groups 
do not exist.

The article is organised  as follows. We give basic requirements in Section~\ref{secPre} including knowledge of Shoda pair, strong Shoda pair and cyclic algebra. In Section~\ref{secII2}, the concepts of nilpotent decomposition, SN group and SSN group are introduced. Several examples are also given. Finite $p$-groups whose non-cyclic subgroups are normal are classified by Z. Bo\v{z}ikov and Z. Janko \cite{BozJan}. We use their result to classify rational group algebras of nilpotent SSN groups that have only one matrix component in Section~\ref{secII4}. In Section~\ref{secII3}, we classify rational group algebras of non-nilpotent SSN groups that have only one matrix component. For one type of such groups, we use S.A. Amitsur's work \cite{Ami} on the classification of finite groups that are embedded in division rings.

\section{Preliminaries}
\label{secPre}

Let $G$ be a finite group. For an element $\alpha = \sum_{g \in G} a_g g$ of $\qq[G]$, 
we denote by $\supp(\alpha) = \{g \in G \mid a_g \neq 0 \}$, the support of $\alpha$.
For a subgroup $H$  of $G$ we put $\widehat{H} = \sum_{h \in H} h \in \zz[G]$ and $\widetilde{H} = \frac{\widehat{H}}{|H|}$ an idempotent in $\qq[G]$. 
If $H=\langle h\rangle$, a cyclic group, then we simply denote $\widehat{H}$ and $\widetilde{H}$ as $\widehat{h}$ and $\widetilde{h}$.
For a normal subgroup $K$ of $H$, we set $$\varepsilon(H, K) = \left\{ \begin{array}{ll} \widetilde{H} & \text{if } H = K \\ \prod_{M/K \in \mathcal{M}(H/K)} (\widetilde{K} - \widetilde{M}) & \text{if } H \neq K , \end{array} \right.$$ where $\mathcal{M}(H/K)$ is the set of all minimal normal nontrivial subgroups of $H/K$. Clearly, $\varepsilon(H, K)$ is a central idempotent of $\qq[H]$. For $\alpha \in \qq[G]$, the {\it $G$-centralizer of $\alpha$} is the subgroup  $\cen_G(\alpha) = \{ g \in G \mid g \alpha = \alpha g\}$. If $T$ is a right transversal of $\cen_G(\varepsilon(H,K))$ in $G$ then  $$e(G, H, K) = \sum_{t \in T} \varepsilon(H, K)^t$$
is a central element of $\qq[G]$ (which is independent of the choice of  $T$).

Olivieri,  del R\'{i}o, and J.J. Sim\'{o}n in \cite{OliRioSim} showed that for a rich class of groups  (for example finite abelian-by-supersolvable groups) these elements (based on (strong) Shoda pairs $(H,K)$) are the primitive central idempotents of $\qq[G]$, hence obtaining a character free description of these elements. We will make intensively use of this and hence we recall the necessary background. For details and more work on this we refer the reader to \cite[Chapter~3]{JesRio}.

A pair $(H, K)$ of subgroups of a group $G$ is called a {\it Shoda pair} if it satisfies the following conditions: $K \trianglelefteq H$, $H / K$ is cyclic and, for every $g \in G \setminus H$, there exists $h \in H$ so that $(h, g) = h^{-1} g^{-1} h g \in H \setminus K$. A pair $(H, K)$ of subgroups of a group $G$ is called a {\it strong Shoda pair} if it satisfies the following conditions: $K \trianglelefteq H \trianglelefteq N_G(K)$, $H / K$ is cyclic and a maximal abelian subgroup of $N_G(K) / K$, and, for every $g \in G \setminus N_G(K)$, $\varepsilon(H,K) g^{-1} \varepsilon(H, K) g = 0$ (one says that the conjugates of the elements $\varepsilon(H,K)$ are orthogonal). It turns out that a pair of subgroups $(H,K)$ is a strong Shoda pair  if and only if it is a Shoda pair with $H \trianglelefteq N_G(K)$ and such that all distinct $G$-conjugates of $\varepsilon(H,K)$ are orthogonal. In this case it follows that  $\cen_G(\varepsilon(H, K)) = N_G(K)$ and $e(G, H, K)$ is a primitive central idempotent of $\qq[G]$.

In the case of finite metabelian groups $G$ the primitive central idempotents are determined by strong Shoda pairs as described in the following result (see \cite[Theorem 3.5.2]{JesRio}). 
This will be next widely used throughout the paper. Recall that a group $G$  is said to be metabelian if it contains a normal abelian subgroup $A$ so that $G/A$ also is abelian.

\begin{thm}[{\cite[Theorem 4.7]{OliRioSim}}]
\label{thmMetabelian}
Let $G$ be a metabelian finite group and let $A$ be a maximal abelian subgroup of $G$ containing $G'$. The primitive central idempotents of $\qq[G]$ are the elements of the form $e(G, H, K)$, where $(H, K)$ is a pair of subgroups of $G$ satisfying the following conditions\textup{:}
\begin{enumerate}
\item[(i)]
$H$ is a maximal element in the set $\{B \leq G \mid A \leq B \text{ and } B' \leq K \leq B\}$ and
\item[(ii)]
$H / K$ is cyclic.
\end{enumerate}
\end{thm}

Note that in Theorem~\ref{thmMetabelian} the subgroup $A$ can be arbitrarily chosen and every pair $(H,K)$ is a strong Shoda pair of $G$.

For a strong Shoda pair $(H,K)$ of a finite group one can also give a description  of the simple Wedderburn component $\qq[G] e(G,H,K)$. This will play an important role in the proofs of our results. Hence, we need to describe the statement in detail and therefore we  have to recall some notation (see  \cite[Section 2.6]{JesRio} and also 
\cite[Section~5]{Pas2}).

Let $R$ be an associative ring with identity and $G$ a group. The multiplicative group of invertible elements of $R$ is denoted $\mathcal{U}(R)$. A {\it crossed product} $R*G$ of   $G$ over $R$ is  an associative ring which has a set of invertible elements $\{u_g: g \in G\}$, a copy of $G$, such that $R*G = \oplus_{g \in G} R u_g$ is a free left $R$-module. The multiplication is determined by the following rules: $u_g r = \alpha_g(r) u_g$  and $u_g u_h = f(g,h) u_{g h}$, for $g, h \in G$, $r \in R$, where $\alpha: G \to \Aut(R)$ and  $f: G \times G \to \mathcal{U}(R)$ are maps and $\alpha_g = \alpha(g)$ (called respectively the {\it action} and the  {\it twisting} of the crossed product). The associativity of $R*G$ is equivalent to the assertions that, for all $g, h, k \in G$, we have $\alpha_g \alpha_h = \iota_{f(g, h)} \alpha_{g h}$  and $f(g, h) f(g h, k) = \alpha_g (f(h, k)) f(g, h k)$, where $\iota_u$ denotes the inner automorphism of $R$ defined by $\iota_u(x) = u x u^{-1}$ for $u \in \mathcal{U}(R)$. We also denote $R*G$ as $(R, G, \alpha, f)$.

Group rings are obvious examples of crossed products;  here the action and twisting are trivial (i.e.  $\alpha_g = 1$ and $f(g, h) = 1$ for all $g, h \in G$).
Another well-known example is a {\it classical crossed product}  $(E/F, f)$, where $E / F$ is a finite Galois extension and $G = \Gal(E/F)$. Here 
we have a natural action  $\alpha: G \to \Aut(E)$ (we identify $\alpha(\sigma)$ and  $\sigma$) and $f : G \times G \to \mathcal{U}(E)$ satisfies  $f(\sigma, \tau) f(\sigma \tau, \rho) = \alpha(\sigma) ( f(\tau, \rho)) f(\sigma, \tau \rho)$, for $\sigma, \tau, \rho \in G$.
If $G = \langle \sigma \rangle$, a cyclic group, say of order $n$, then one can take for $f$ the mapping defined by $f(\sigma^i, \sigma^j) = 1$ if $i+j<n$ and  $f(\sigma^i, \sigma^j) = a$ otherwise, where $0 \leq i, j < n$ and  $a \in \mathcal{U}(F)$. A crossed product of this type is called a {\it cyclic algebra} and  is denoted by $(E/F, \sigma, a)$, or simply as $(E, \sigma, a)$ or $(E/F, a)$ if $\sigma$ is clear from the context.

Recall that an $F$-algebra is called a {\it central algebra} if its center is exactly the field $F$. For our investigations we will make use of the following well-known properties (see for example {\cite[Theorem~2.6.3 and Proposition~2.6.7]{JesRio}}).

\begin{prop} $\;$
\label{propCyclicAlg}
\begin{enumerate}
\item[(i)]
Every classical crossed product $(E/F, f)$ is a finite dimensional central simple $F$-algebra.
\item[(ii)]
A cyclic algebra $(E/F, \sigma, a) \simeq M_n(F)$ if and only if $a \in N_{E/F}(E)$ where $n = |\Gal(E/F)|$ and $N_{E/F}(E)$ consists elements of the norm $N_{E/F}(x) = \prod_{\sigma\in \Gal(E/F)} \sigma(x)$ for $x \in E$. In particular, $(E/F, \sigma, 1) \simeq M_n(F)$. 
\end{enumerate}
\end{prop}

We are now in a position to describe the  Wedderburn component associated to a strong Shoda pair (see for example Theorem 3.5.5 in \cite{JesRio}). By $\xi_n$ we denote  a primitive $n$-th root of $1$ in $\cc$.

\begin{prop}[{\cite[Proposition 3.4]{OliRioSim}}]
\label{propSShodaPair}
Let $(H, K)$ be a strong Shoda pair of a finite group $G$. Let $h = [H:K]$, $N = N_G(K)$, $n = [G: N]$, $H / K = \langle x \rangle$. Denote the (conjugation) action of $N / H \simeq (N / K) / (H / K)$ on $H / K$ by $(x^i)^a$ for $a \in N / H$ and $x^i \in H/K$. Then $$\qq[G] e(G, H, K) \simeq M_n(\qq(\xi_h) * N/H),$$ where the action and twisting are given by 
   $$\begin{array}{lll} \text{action\textup{:}} & \xi_h^a = \xi_h^i & \text{if } x^a = x^i  ,\\ 
                                  \text{twisting\textup{:}} & f(a, b) = \xi_h^j & \text{if } a' b' = x^j (a b)'   ,\end{array}$$ 
for $a, b \in N / H$ and integers $i, j$, and  where $a' , b', (ab)'$ are fixed preimages of $a, b, ab$ under the natural homomorphism $N/K \to N/H$.
\end{prop}

To investigate the ND property we only need to consider the Wedderburn components of reduced degree larger than $1$, in particular they need to be non-commutative.
The following is an easy criterium to determine these in terms of strong Shoda pairs.

\begin{lem}
\label{lemSSPcomm}
Let $(H,K)$ be a strong Shoda pair of $G$ and thus  $e(G,H,K)$ is primitive central idempotent of $\qq[G]$. Then, the Wedderburn component $\qq[G] e(G,H,K)$ is commutative if and only if $H = G$ if and only if $K$ contains the commutator subgroup $G'$.
\end{lem}

\begin{proof}
By \cite[Proposition 3.4]{OliRioSim} (see also Proposition~\ref{propSShodaPair}), $\qq[G] e(G,H,K)$ is commutative if and only if $H = G$ (for the ``only if'' direction, one has that $H / K$ is central in $N_G(K) / K$ and $N_G(K) = G$). If $H = G$, then $G' = H' \subseteq K$ because $H / K$ is abelian. If $G' \subseteq K$, then $N_G(K) = G$ and $G / K$ is abelian. Hence $H = G$ because $H / K$ is maximal abelian in $N_G(K) / K = G / K$. The result follows.
\end{proof}


\section{Nilpotent Decomposition Property, \textup{SN} and \textup{SSN} Groups}
\label{secII2}

Let $G$ be a finite group. If   $\qq[G] = A_1 \oplus A_2 \oplus \cdots \oplus A_m$, where each $A_i \simeq M_{s_i}(D_i)$ for some division ring $D_i$, the Wedderburn decomposition, then every nilpotent element $\alpha$ of $\zz[G]$ has the Wedderburn decomposition $\alpha = \sum_i \alpha_i$, where $\alpha_i$ is a nilpotent element of $A_i$. In general $\alpha_i$ does not have to belong to $\zz [G]$.  Therefore the following definition.

\begin{defn}
\label{defnND}
A nilpotent element $\alpha$ of $\zz[G]$ is said to have the {\it nilpotent decomposition property} (ND) if  $\alpha e \in \zz[G]$ for every central idempotent $e \in \qq[G]$. An integral group ring $\zz[G]$, or simply  the group $G$, is said to have the {\it nilpotent decomposition property} (ND) if every nilpotent element of $\zz[G]$ has ND. 
\end{defn}

For two relatively prime positive integers we denote by $\ord_n(m)$ the multiplicative order of $m$ modulo $n$. Clearly, we have the following observation.
Recall that we say that a Wedderburn component of  $\qq[G]$ is a {\it matrix component} if it has reduced degree larger than one.

\begin{lem}
\label{lemOneMatrix}
For a finite group $G$, if $\qq[G]$ has at most one Wedderburn component which is a matrix component then $\zz[G]$ has \ND{}. Thus, the following Dedekind groups have \textup{ND:}
\begin{enumerate}
\item[(i)]
abelian groups\textup{;}
\item[(ii)]
$Q_8 \times E \times A$ where $E$ is an elementary abelian $2$-group and $A$ is an abelian group of odd order $m$ such that $\ord_2(m)$ is odd\textup{;}
\item[(iii)]
$Q_8 \times C_p$ where $p$ is an odd prime such that $\ord_2(p)$ is even.
\end{enumerate}
In particular, $\qq[G]$ does not have  any matrix component for groups $G$ of types \textup{(i)} and \textup{(ii)}.
\end{lem}

All the groups of order at most $11$ satisfy the hypothesis of Lemma~\ref{lemOneMatrix} and hence they have property ND. Also the alternating group $A_4$ of degree 4 and the generalized quaternion group $Q_{12}$ of order $12$ satisfy the hypothesis of Lemma~\ref{lemOneMatrix}, and thus they have ND, since 
$\qq[A_4] \simeq \qq \oplus \qq(\xi_3) \oplus M_3(\qq)$ and 
$\qq[Q_{12}] \simeq 2 \qq \oplus \qq(\sqrt{-1}) \oplus \left( \frac{-3, -1}{\qq} \right) \oplus M_2(\qq)$, where $\left( \frac{-3, -1}{\qq} \right)$ is the quaternion division algebra over $\qq$ defined by $-3$ and $-1$  (see for example \cite[Example 2.1.7(3)]{JesRio}). The group $D_{12} = \langle a, b \mid a^6 = b^2 = 1, b a b^{-1} = a^{-1} \rangle$, the dihedral group of order 12, does not have ND. Indeed, 
 $\alpha = (1-b) a (1+b) \in \zz[D_{12}]$ is a nilpotent element and $e = \widetilde{a^3} - \widetilde{a} \in \qq[D_{12}]$ is a central idempotent but $\alpha e = \frac{1}{2} (1+a^3) (a - a^{-1}) (1 + b) \not\in \zz[D_{12}]$.

Another useful observation os the following.

\begin{lem}[{\cite[Corollary 12 and Corollary 13]{HalPasWil}}]
\label{lemDecNilp}
Let $N$ be a nontrivial normal subgroup of  a nonabelian finite group $G$. Assume that $G$ has \ND{}. If $H$ is a subgroup of $G$ such that $H \cap N = 1$, then $\qq[H]$ has no nonzero nilpotent elements.
In particular, if $H$ is a finite group such that $\qq[H]$ has a nonzero nilpotent element, then the group $H \times N$ does not have \textup{ND} for any nontrivial finite group $N$.
\end{lem}

\begin{proof}
Consider the central idempotent  $e =\widetilde{N}=\frac{1}{|N|}\sum_{n\in N}n\in \qq[G]$. If $\qq[H]$ has some nonzero nilpotent element, then there exists a nilpotent element $\alpha$ in $\zz[H] \setminus |N| \zz[H]$. However, $H \cap N = 1$ implies that $N h \cap N h' = \varnothing$ for $h \neq h' \in H$ and thus it easily follows that $e \alpha \not\in \zz[G]$, a contradiction.
\end{proof}

The aim is to describe the  finite groups that have property  ND. Note that we have a natural way to find nilpotent elements in $\zz[G]$, namely $$(1-y) g \widehat{Y} \quad \text{and} \quad \widehat{Y} g (1-y),$$ for every subgroup $Y$ of $G$, $y \in Y$ and $g \in G$. Using these nilpotent elements, 
one shows that the ND property leads to the class of groups with property SN, as introduced by Liu and Passman.

\begin{prop}[{\cite[Proposition 2.5]{LiuPas}}]
\label{propDecNilp}
Let $G$ be  a finite group. If $G$ has property \ND{} then $G$ satisfies property \SN, that is, 
if  $N \trianglelefteq G$ and  $Y\leq G$, then  $Y \supseteq N$ or $Y N \trianglelefteq G$.
\end{prop}

\begin{proof}
Suppose $H = Y N$ is not normal in $G$. We claim that $Y \supseteq N$. Let $g \in G$ be such that $H^g \neq H$. Then $Y^g \not\subseteq H$. Choose $y \in Y$ with $y^g \not\in H$. Consider the nilpotent element $\alpha = (1-y) g \widehat{Y} \in \zz[G]$ and central idempotent  $e =\widetilde{N}\in \qq[G]$. Because, $G$ has ND, we have $\alpha e \in \zz[G]$. Note that $\widehat{Y} \widehat{N} = \widehat{H} |Y \cap N|$ and $\alpha e = (1-y) g \widehat{H} |Y \cap N| / |N|$. As $y^g \not\in H$, we have that the two cosets $g H$ and $y g H$ are disjoint. Hence, we must have $|Y \cap N| / |N| \in \zz$ and thus $N = Y \cap N \subseteq Y$.
\end{proof}

The cyclic group of order $n$ we denote by $C_n$ and if $g$ is a generator then we also denote it as $\langle g \rangle_{n}$.
The nonabelian group $C_3 \rtimes C_8$ has SN but it does not have property ND (see for example the proof of \cite[Lemma 2.3]{LiuPas2} or \cite[Theorem 4.1]{HalPas2})
 We have another counterexample in Example~\ref{exmpSSN3}. So the class of SN groups is larger than the class of ND groups.

Note that the dihedral group $D_{12}=\langle a\rangle_6 \rtimes \langle b\rangle_{2}$ does not have SN by considering $Y = \langle b \rangle$ and $N = \langle a^3 \rangle$. Thus, $D_{12}$ does not have ND by Proposition~\ref{propDecNilp}. This gives another easy argument for this fact.

It is easy to check that the SN property is inherited by quotient groups. However, the ND property may be not inherited by quotient groups because in general nilpotent elements of an integral group ring can not be lifted via an epimorphism of groups, see \cite[p. 117]{HalPasWil} for an example.
However, subgroups of SN groups need not be SN (see Example~\ref{exmpSSN2}). Hence, the following definition introduced by Liu and Passman \cite{LiuPas5}.

\begin{defn}
\label{defnSSN}
A finite group $G$ is said to have SSN if every subgroup of $G$ has SN.
\end{defn}

That the classes of SN groups and SNN groups are distinct  follows for example from  Example~\ref{exmpSSN2}. 
We now present the classification of finite SSN groups obtained in \cite{LiuPas5}.
It will be used in Section~\ref{secII4} and Section~\ref{secII3}. We begin with nilpotent groups.

\begin{prop}[{\cite[Proposition 2.2 and Proposition 2.3]{LiuPas5}}]
\label{propSSN1}
Let $G$ be a finite nilpotent group.
\begin{enumerate}
\item If $G$ is not a $p$-group for any prime $p$ then $G$ has \textup{SSN} if and only if $G$ is either abelian or Hamiltonian.
\item If $G$ is a $p$-group for some prime $p$ then $G$ has  \textup{SSN} if and only if all non-cyclic subgroups of $G$ are normal (one calls this the \textup{NCN} property).
\end{enumerate}
\end{prop}

A complete classification of finite $p$-groups with NCN has been given by  Bo\v{z}ikov and Janko in \cite[Theorem 1.1]{BozJan}. 
For a formulation we refer to  Theorem~\ref{thmNCN2}. 

Recall that a group $Q$ is said to act {\it reducibly} on a group $P$ if $P = P_1 \times P_2$ for some nontrivial $Q$-stable subgroups $P_1$ and $P_2$ of $P$. Otherwise, $Q$ {\it acts irreducibly} on $P$. For solvable SSN groups that are not nilpotent, one has the following result.

\begin{thm}[{\cite[Theorem 2.7]{LiuPas5}}]
\label{thmSSN}
Let $G$ be a finite solvable group that is not nilpotent. Then $G$ has \textup{SSN} if and only if $G$ is one of the following two types of groups\textup{:}
\begin{enumerate}
\item[(i)]
$G = P \rtimes Q$ where $P$ is a normal elementary abelian $p$-subgroup of $G$ for some prime $p$, $Q$ is a cyclic $p'$-subgroup of $G$ which acts faithfully on $P$, and every nontrivial subgroup of $Q$ acts irreducibly on $P$.
\item[(ii)]
$G = P \rtimes Q$ where $P$ is a normal subgroup of $G$ of order $p$ for some prime $p$, $Q$ is a cyclic $q$-group for some prime $q \neq p$ with $|Q| \geq q^2$, and $Q$ does not act faithfully on $P$.
\end{enumerate}
\end{thm}

There is only one non-solvable SSN group.

\begin{thm}[{\cite[Theorem 3.3 and Corollary 3.4]{LiuPas5}}]
\label{thmSSN2}
The alternating group $A_5$ of degree $5$ is the unique nonabelian simple group with \textup{SSN}. Moreover, $A_5$ is the unique non-solvable finite group with \textup{SSN}.
\end{thm}

We now give some concrete examples of SSN groups that satisfy Theorem~\ref{thmSSN}(i). Not all of them have property ND.

\begin{exmp}
\label{exmpSSN}
(i) $A_4 = (C_2 \times C_2) \rtimes C_3 = \langle a, b, c \mid a^2 = b^2 = c^3 = 1, a b = b a, a^c = b, b^c = a b \rangle$ has SSN and ND (since $\qq[A_4]$ has only one matrix component).

(ii) $G = C_5 \rtimes C_4 = \langle a, b \mid a^{5} = b^4 = 1, b a b^{-1} = a^2 \rangle$  has SSN and ND because $\qq[G] \simeq 2 \qq \oplus \qq(i) \oplus M_4(\qq)$ (see  Lemma~\ref{lemDecNilp8}).

(iii) $G = C_{11} \rtimes C_5 = \langle a, b \mid a^{11} = b^5 = 1, b a b^{-1} = a^3 \rangle$  has SSN and ND because 
 $\qq[G] \simeq \qq \oplus \qq(\xi_5) \oplus M_5(\qq(\alpha))$ (see Lemma~\ref{lemDecNilp8}), where $\alpha = \xi_{11} + \xi_{11}^3 + \xi_{11}^4 + \xi_{11}^5 + \xi_{11}^9$ and $\qq(\alpha) = \qq(\sqrt{-11})$.
\end{exmp}

\begin{exmp}
\label{exmpSSN2}
We give an example to illustrate that there is a finite SN group which does not have SSN. This example also shows that the ND property is not necessarily inherited by subgroups. Let $$G = (C_3 \times C_3) \rtimes C_8 = \langle a, b, c \mid a^3 = b^3 = c^8 = 1, a b = b a, a^c = b, b^c = a b \rangle.$$ 
It is easy to check that each nontrivial normal subgroup $N$ of $G$ must contain $G'=\langle a, b \rangle$. Thus, $G / N$ is cyclic and then $Y N \trianglelefteq G$ for every subgroup $Y \leq G$. So $G$ has SN. Consider the subgroup $G_1 = \langle a, b, c^4 \rangle = (C_3 \times C_3) \rtimes C_2$. 
Since $a^{c^4} = a^2$, $N_1 = \langle a \rangle$ is a normal subgroup of $G_1$. Let $Y = \langle b c^4 \rangle$. Then $Y \not\supseteq N_1$. Moreover, $(b c^4)^{c^4} = b^2 c^4 \not\in Y N_1$. Thus, $G_1$ does not have SN and then $G$ does not have SSN. 

From Proposition~\ref{propDecNilp} we also conclude that  $G_1$ does not have ND (note that $\qq[G_1] \simeq 2 \qq \oplus 4 M_2(\qq)$ has four matrix components, see for example \cite[Group Type 18/5 on p. 193]{Gir}). However, $G$ has ND since $\qq[G]$ has only one matrix component. To see this, we first note that $G' = \langle a, b \rangle$ is a maximal abelian subgroup of $G$ containing $G'$ and that 
$(H,K)=(G',\langle a \rangle )$ 
is a pair satisfying conditions in Theorem~\ref{thmMetabelian}. Thus, $e(G, H, K)$ is a primitive central idempotent of $\qq[G]$ since $G$ is metabelian. One can check that $e(G, H, K) = 1 - \widetilde{G'}$. It follows that  $\qq[G]$ has only one non-commutative component. So $G$ indeed has ND. 

Further note that $(H, K)$ is a strong Shoda pair of $G$ with  $N = N_G(K) = \langle a, b, c^4 \rangle$. Hence, by Proposition~\ref{propSShodaPair}, $\qq[G] e(G, H, K) \simeq M_4(A)$ where $A = \qq(\xi_3) * N / H$. Since $N / H = \langle \overline{c^4} \rangle$ and $a^{c^4} = a^2$, the action is $\xi_3^{\overline{c^4}} = \xi_3^2$. Fix preimages $K1$ and $K c^4$ of $\overline{1}$ and $\overline{c^4}$, respectively, under the natural homomorphism $N / K \to N / H$. The twisting is $f(\overline{c^4}^i, \overline{c^4}^j) = 1$ for $0 \leq i, j \leq 1$. Note that the fixed field of $N/K$ in $\qq(\xi_3)$ is $\qq$. It follows that $A$ is actually a cyclic algebra $(\qq(\xi_3) / \qq, \overline{c^4}, 1)$. Thus, by Proposition~\ref{propCyclicAlg}(ii), we have $A \simeq M_2(\qq)$. Hence, $\qq[G] e(G,H,K) \simeq M_4(A) \simeq M_8(\qq)$. 
\end{exmp}

\begin{exmp}
\label{exmpSSN3}
The group $G$ in  Example~\ref{exmpSSN2} has SN and ND. We now show that the subgroup 
$K = \langle a, b, c^2 \rangle$ has SN but does not have ND.
Observe that $K' = \langle a, b \rangle$ 
and that every nontrivial normal subgroup of $K$ contains $K'$. Thus, $K$ has SN.
To prove that $K$ does not have ND consider the nilpotent element  $\alpha = (1-c^4) a (1+c^4) \in \zz[K]$.
Note that $e = e(K, K', \langle a \rangle)$ is a primitive central idempotent of $\qq[K]$ by Theorem~\ref{thmMetabelian}. 
 Since $\cen_g(\varepsilon(K', \langle a \rangle)) = \langle a, b, c^4 \rangle$, we have $e = \varepsilon(K', \langle a \rangle) + \varepsilon(K', \langle a \rangle)^{c^2} 
 = \widetilde{a} + \widetilde{a b} - 2 \widetilde{a} \widetilde{b}$. Note that  $\alpha = a (1 - a c^4) (1 + c^4) = a (1 - a) \widehat{c^4}$
 and thus  $e \alpha = \widetilde{a b} \, a (1 - a) \widehat{c^4} \not\in \zz[K]$. Hence $K$ is not ND.
\end{exmp}

\begin{remk}\label{remkSN}
The reader can consult Figure~\ref{figure1} at  the end of this article for a concise summary in relations between MJD, ND, SN and SSN. 
The following examples are used in this context.\\
\indent  (i) 
Note that groups $S_3$, $D_8$, $Q_8$, $D_{10}$, $Q_{12}$ and $A_4$ have SN since they have ND. The group of smallest order that is a non-SN group is $D_{12}$. Note that all groups of order $24$ containing $D_{12}$ are $D_{24} \simeq C_4 \rtimes S_3$, $C_3 \rtimes D_8$, $C_4 \times S_3$ and $C_2 \times C_2 \times S_3$. These groups do not have SN since, with notation as in Proposition~\ref{propDecNilp}, we can choose subgroups $N$ and $Y$ such that
 such that $N \cap Y = 1$ and $N Y$ is not normal. So all SN groups of order $24$ must have SSN. The only nonabelian group of order $14$ is $D_{14}$ which has SSN since it is a group in Theorem~\ref{thmSSN}(i). Groups of order $15$ are abelian. The group $C_2 \times D_8$ does not have SN by choosing $N = C_2$ and $Y \leq D_8$ such that $Y \not\trianglelefteq D_8$. Moreover, the central product (\cite[p. 29]{Gor}) $D_8 \Ydown D_8$ of order $32$ has ND since $\qq[D_8 \Ydown D_8] \simeq 16 \qq \oplus M_4(\qq)$ has only one matrix component, see \cite[Group Type 32/42 on p. 199]{Gir}. Thus, $D_8 \Ydown D_8$ has SN. However, this group has a subgroup isomorphic to $C_2 \times D_8$. As a consequence, $D_8 \Ydown D_8$ is the smallest SN group which does not have SSN. It is also the smallest non-SSN group such that its rational group algebra has only one matrix component.\\
\indent (ii) If $G$ is a nonabelian SN group with $|G| \leq 22$, then one can verify (via a  case-by-case argument)  that $\qq[G]$ has only one matrix component. The group $C_3 \rtimes C_8$ with a nontrivial action is an SSN group by Theorem~\ref{thmSSN}(ii) and it does not have ND by the proof of \cite[Lemma 2.3]{LiuPas2} or \cite[Theorem 4.1]{HalPas2}. In other words, $C_3 \rtimes C_8$ is a group of smallest order that has SN but which does not have ND.
\end{remk}



\vspace{12pt}
\section{Nilpotent Decomposition for Nilpotent \textup{SSN} Groups}
\label{secII4}

In this section, we study finite nilpotent groups $G$ that have SSN. This with the aim of proving our first main result: a description of such groups $G$ so that $\qq G$ has only one matrix component. In the first subsection we give a series of lemmas that deal with  $p$-groups, in the second subsection we handle with the groups that are not $p$-groups and in the third subsection we state the conclusion, the first main result.


\subsection{Finite SSN groups that are $p$-Groups} $\;$
\label{secII41}

By Proposition~\ref{propSSN1}, a finite  $p$-group has the SSN property if and only if it has the NCN property, i.e.  all its non-cyclic subgroups are normal.
Since  rational group algebras of Hamiltonian and abelian  $2$-groups have no matrix components, in this section, we only have to deal with non-Hamiltonian  NCN groups
that are nonabelian.
Non-Hamiltonian NCN finite $p$-groups had been classified by Z. Bo\v{z}ikov and Z. Janko \cite{BozJan}. In order to state their result we recall some terminology and notation.
A $p$-group is called {\it minimal nonabelian} (resp. {\it minimal non-metacyclic}) if it is nonabelian (resp. non-metacyclic) and all proper subgroups are abelian (resp. metacyclic). A $p$-group $G$ is said to be  of {\it maximal class} if $G$ has nilpotency class $n$ where $|G| = p^{n+1}$. Moreover, the notation $\Omega_1(G)$ is defined by $\Omega_1(G) = \langle g \in G \mid g^{p} = 1 \rangle$. 

\begin{thm}[{\cite[Theorem 1.1]{BozJan}}]
\label{thmNCN2}
Let $G$ be a finite $p$-group. Then all the non-cyclic subgroups of $G$ are normal if and only if $G$ is either abelian, Hamiltonian or satisfies one of the following conditions\textup{:}
\begin{enumerate}
\item[(BJ1)]
$G$ is metacyclic minimal nonabelian and $G$ is not isomorphic to $Q_8$, the quaternion group of order $8$.
\item[(BJ2)]
$G = G_0 \Ydown Z$, the central product of a nonabelian group $G_0$ of order $p^3$ with a cyclic group $Z$, where $G_0 \cap Z = \mathcal{Z}(G_0)$, the center of $G_0$, and if $p=2$, then $|Z| > 2$.
\item[(BJ3)]
$p=2$ and $G = Q_8 \times Z$ where $Z$ is cyclic of order $>2$.
\item[(BJ4)]
$G$ is a group of order $3^4$ and of maximal class with $\Omega_1(G) = G' \simeq C_3 \times C_3$.
\item[(BJ5)]
$G = \langle a, b \mid a^8 = 1, a^b = a^{-1}, a^4 = b^4 \rangle$.
\item[(BJ6)]
$G = Q_{16}$.
\item[(BJ7)]
$G = D_8 \Ydown Q_8$.
\item[(BJ8)]
$G = \langle a, b, c \mid a^4 = b^4 = [a, b] = 1, c^2 = a^2, a^c = a b^2, b^c = b a^2 \rangle$, the minimal non-metacyclic group of order $2^5$.
\item[(BJ9)]
$G = \langle a, b, c, d \mid a^4 = b^4 = [a, b] = 1, c^2 = a^2 b^2 , a^c = a^{-1}, b^c = a^2 b^{-1}, d^2 = a^2, a^d = a^{-1} b^2, b^d = b^{-1}, [c, d] = 1 \rangle$, a special\footnote{A $p$-group $P$ is called {\it special} if either it is elementary abelian or $P' = \mathcal{Z}(P) = \Phi(P)$ is elementary abelian where $\Phi(P)$ is the Frattini subgroup of $P$. 
} $2$-group of order $2^6$ in which every maximal subgroup is isomorphic to the minimal non-metacyclic group of order $2^5$ of type  (BJ8).
\end{enumerate}
Conversely, all the above groups satisfy the assumptions of the theorem.
\end{thm}

So, in order to deal with finite  $p$-groups that have  SSN we will have to deal separately with each type listed  in the Theorem. To do so we 
we will make  use of  Theorem~\ref{thmMetabelian} that describes Wedderburn components via strong Shoda pairs $(H,K)$. Recall from Lemma~\ref{lemSSPcomm} 
that the corresponding component is not commutative when $K \not\supseteq G'$.\\

{\it Groups of type \textup{(BJ1)}}\\

According to \cite[Proposition 1.3]{BozJan} or \cite[Lemma 3.1]{BerJan}, groups of type  (BJ1) have the following presentation (proved by L. R\'{e}dei) $$G = \langle a, b \mid a^{p^m} = b^{p^n} = 1, \, a^b = a^{1+p^{m-1}} \rangle = C_{p^m} \rtimes C_{p^n},$$ where $m \geq 2$, $n \geq 1$, $|G| = p^{m+n}$ and $|G'| = p$. 
Note that 
$a^p$ and $b^p$ are central,  and  $G' = \langle a^{p^{m-1}} \rangle$ since $[a, b] = a^{p^{m-1}}$ and $|G'| = p$. We remark here that $G$ has ND when $(p, m, n) = (2, 2, 1), (2, 2, 2), (2, 3, 1)$ and $(3, 2, 1)$, since they have MJD by \cite[Lemma~2.10]{Liu} and \cite[Lemma~3.6]{LiuPas3}. Moreover, the proof of \cite[Lemma~2.8]{Liu} shows that $G$ does not have ND for $p=2$ and $(m, n) = (2, \, \geq 4)$, $(3, \, \geq 2)$. Now we discuss when $\qq[G]$ has only one matrix component.

\begin{lem}
\label{lemDecNilpBJ1}
Let $n = 1$. Then $\qq[G]$ has only one matrix component for any $p, m \geq 2$. Moreover, this component is isomorphic to $M_p(\qq(\xi_{p^{m-1}}))$.
\end{lem}

\begin{proof}
Let $A = \langle a \rangle$. Then $A$ is of index $p$ in $G$. Thus, $A$ is a maximal abelian subgroup of $G$ containing $G' = \langle a^{p^{m-1}} \rangle$. Let $(H,K)$ be a pair of subgroups of $G$ in Theorem~\ref{thmMetabelian} such that $G' \not\subseteq K$. Then $(H,K) = (A,1)$. In other words, $e(G, A, 1) = \varepsilon(A, 1) = 1 - \widetilde{a^{p^{m-1}}} = 1 - \widetilde{G'}$ is the only primitive central idempotent $e$ such that $\qq[G] e$ is non-commutative.

To describe  $\qq[G] e(G, A, 1)$,  note that $(A, 1)$ is a strong Shoda pair of $G$. By Proposition~\ref{propSShodaPair}, $\qq[G] e(G, A, 1) \simeq \qq(\xi_{|a|}) * (G / A) \simeq \qq(\xi_{p^m}) * \langle \sigma \rangle$ where $\sigma \in \Aut(\qq(\xi_{p^m}))$ and $\sigma(\xi_{p^m}) = \xi_{p^m}^{1+p^{m-1}}$. Note that $\sigma$ fixes $\xi_{p^{m-1}}$ since $\xi_{p^{m-1}} = \xi_{p^m}^p$. It is easy to check that the fixed field $\qq(\xi_{p^m})^{\sigma} = \qq(\xi_{p^{m-1}})$. Note also that $\sigma^p$ is the identity map because $(1+p^{m-1})^p \equiv 1 \pmod{p^m}$. By Proposition~\ref{propCyclicAlg}, we can conclude that $\qq(\xi_{p^m}) * \langle \sigma \rangle \simeq M_p(\qq(\xi_{p^{m-1}}))$.
\end{proof}

\begin{lem}
\label{lemDecNilpBJ1-2}
Let $n \geq 2$. Then $\qq[G]$ has only one matrix component if and only if $p = m = n = 2$. Moreover, this component is isomorphic to $M_2(\qq)$.
\end{lem}

\begin{proof}
First of all, if $p = m = n = 2$, then $$\qq[G] \simeq 4 \qq \oplus 2 \qq(i) \oplus \hh_{\qq} \oplus M_2(\qq)$$ by \cite[Group Type 16/10, p. 195]{Gir} which has exactly one matrix component. Assume that $\qq[G]$ has only one matrix component. We will prove that $p = m = n = 2$. Let $A = \langle a, b^p \rangle = \langle a \rangle \times \langle b^p \rangle$. Then $A$ is of index $p$ in $G$ and it is a maximal abelian subgroup of $G$ containing $G' = \langle a^{p^{m-1}} \rangle$. 
Consider the subgroups  $K_i = \langle a^{i p^{m-1}} b^p \rangle$ for $i = 0, 1, \ldots, p-1$. Note that $G' \not\subseteq K_i$ because $n\geq 2$. Also note that each $(A, K_i)$ is a strong Shoda pair by Theorem~\ref{thmMetabelian}. Moreover, $K_i$ is central and $A$ is normal in $G$. It follows that $A^g \cap K_i \neq K_j^g \cap A$ for $i \neq j$ and every $g \in G$. By \cite[Problem~3.4.3]{JesRio}, all $\qq[G] e(G, A, K_i)$ are distinct non-commutative Wedderburn components. Because $G$ is a $p$-group, Roquette's theorem (see \cite{Roq} or \cite[Corollary~13.5.5]{JesRio}) implies that the Schur index of each Wedderburn component of $\qq[G]$ is at most $2$. If $p$ is odd, then the Schur index must be $1$ since it divides the order of $G$. So every non-commutative Wedderburn component is a matrix component. In this case, $\qq[G]$ has at least $p$ matrix components. Hence $p=2$.

Note that $e(G, A, K_0) = \widetilde{b^2} (1 - \widetilde{G'})$ and then one has $\qq[G] e(G, A, K_0) \simeq (1 - \widetilde{\overline{G}'}) \qq[\overline{G}] \simeq M_2(\qq(\xi_{2^{m-1}}))$ by Lemma~\ref{lemDecNilpBJ1} where $\overline{G} = G / \langle b^2  = \langle \overline{a} \rangle_{2^m} \rtimes \langle \overline{b} \rangle_{2}$ is a group of type (BJ1). Thus, by the assumption that $\qq[G]$ has only one matrix component, $\qq[G] e(G, A, K_1)$ is a non-commutative division algebra and we denote it as $D$. By \cite[Corollary~13.5.4]{JesRio}, $D$ is a totally definite quaternion algebra. Thus the center of $D$ does not contain roots of unity of order greater than $2$. Note that $e(G, A, K_1) = \widetilde{K_1} - \widetilde{A}$ and, moreover, $a^{2^k} \widetilde{K_1} \neq \widetilde{K_1}$ for $0 < k < m$ and $b^{2^l} \widetilde{K_1} \neq \widetilde{K_1}$ for $0 < l < n$. Thus $a^{2^k} e(G, A, K_1) \neq e(G, A, K_1)$ and $b^{2^l} e(G, A, K_1) \neq e(G, A, K_1)$ since $A$ contains $a^2$ and $b^2$. As a consequence, $a^2 e(G, A, K_1)$ and $b^2 e(G, A, K_1)$ are two central elements of $D$ of orders $2^{m-1}$ and $2^{n-1}$,  respectively. It follows that $m = n = 2$.
\end{proof}

The two previous lemmas yield the following result.

\begin{prop}
\label{propDecNilBJ1}
Let $G = \langle a, b \mid a^{p^m} = b^{p^n} = 1, \, a^b = a^{1+p^{m-1}} \rangle$ be of type \textup{(BJ1)}, where $p$ is a prime, $m \geq 2$ and $n \geq 1$. Then, $\qq[G]$ has only one matrix component if and only if either $n=1$ or $p = m = n = 2$. In particular, the matrix component is isomorphic to $M_p(\qq(\xi_{p^{m-1}}))$.
\end{prop}

\vspace{12pt}
{\it Groups of type \textup{(BJ2)}}\\

Let $G = G_0 \Ydown Z$, the central product of a nonabelian group $G_0$ of order $p^3$ and  a cyclic $p$-group $Z$  such that $G_0 \cap Z = \mathcal{Z}(G_0)$, and $|Z| > 2$ if $p=2$. Note that $G' = G_0' = \mathcal{Z}(G_0) \simeq C_p$ and $G_0 / \mathcal{Z}(G_0) \simeq C_p \times C_p$. Write $G_0 / \mathcal{Z}(G_0) = \langle \overline{a} \rangle \times \langle \overline{b} \rangle$ for $a, b \in G_0$. Note also that one can choose $a$ such that $|a| = p$. Then $a^b = b a b^{-1} = a z_0$ for some $z_0 \in \mathcal{Z}(G_0)$. If $z_0 = 1$, then $G$ is abelian because $G = \langle a, b, \mathcal{Z}(G_0) \rangle$. Thus, $z_0 \neq 1$ and $\mathcal{Z}(G_0) = \langle z_0 \rangle$. 

\begin{lem}
\label{lemDecNilpBJ2}
Let $G$ be a group of type \textup{(BJ2)}. Then $\qq[G]$ has only one matrix component which is isomorphic to $M_p(\qq(\xi_{|Z|}))$.
\end{lem}

\begin{proof}
Let $G = G_0 \Ydown Z$, $z_0$, $a$, $b$ be as above and let $A = \langle a \rangle Z$. Then $A = \langle a \rangle_p \times \langle z \rangle_{p^n}$ is of index $p$ in $G$ where $Z = \langle z \rangle$ such that $z^{p^{n-1}} = z_0$. In other words, $A$ is a maximal abelian subgroup of $G$ containing $G' = \langle z_0 \rangle$. 
Moreover, $G$ is metabelian since $G' $ is abelian.
We will show that $\qq[G]$ has only one primitive central idempotent corresponding to a non-commutative Wedderburn component.   Let $(H,K)$ be a pair of subgroups of $G$ as in Theorem~\ref{thmMetabelian} such that $G' \not\subseteq K$, i.e., because of Lemma~\ref{lemSSPcomm}, a pair that does determine a non-commutative simple component. Since $A$ is of index $p$ in $G$ and $K \not\supseteq G'$, we can conclude that $H = A$ and $A / K$ is cyclic. 
 It follows that $K$ must be one of the groups $K_i = \langle a z^{i p^{n-1}} \rangle$ for $i = 0, 1, \ldots, p-1$. Recall that $G_0 = \langle a, b \rangle$ with $a^b = a z_0 = a z^{p^{n-1}}$. This implies that the groups $K_i$ are conjugate. So all the $e(G, A, K_i)$ are equal. Thus $\qq[G]$ has a unique non-commutative Wedderburn component. Furthermore, one has $N_G(K) = A$, $|A| = p |Z|$ and $|K| = p$. By Proposition~\ref{propSShodaPair}, $\qq[G] e(G, A, K) \simeq M_p(\qq(\xi_{|Z|}))$.
\end{proof}

\vspace{12pt}
{\it Groups of type \textup{(BJ3)}}\\

Note that $\qq[Q_8 \times C_4] \simeq 2(4 \qq \oplus \hh_{\qq}) \oplus 4 \qq(\sqrt{-1}) \oplus M_2(\qq(\sqrt{-1}))$ has only one matrix component. The following result imitates from \cite[Lemma 2.13]{Liu} that $Q_8 \times C_8$ does not have ND. For convenience, we denote by $\equiv_n$ the equivalence relation on $\zz[G]$ modulo $n$ for $n \in \nn$. In other words, $\alpha \equiv_n \beta$ if $\alpha - \beta \in n \zz[G]$.

\begin{lem}
\label{lemDecNilpBJ3}
$G = Q_8 \times C_{2^n}$ does not have \textup{ND} for $n \geq 3$ and hence $\qq[G]$ has more than one matrix component.
\end{lem}

\begin{proof}
Let $Q_8 = \langle a, b \mid a^4 = 1, b^2 = a^2, a^b = a^{-1}\rangle$ and $C_{2^n} = \langle x \rangle$. Write $z = a^2 = b^2$. Then $b a = a b z$. Note that $z, t$ are central. Denote $t = x^{2^{n-3}}$. Let $$r = (a + b t) ( 1 - t^2) (1 + t^4) (1 - z) \quad \text{and} \quad s = (a + b t^2) (1 - t^4) (1 - z).$$ Clearly, $r s = s r = 0$ since $t^8 = 1$. Observe that $(a+b t)^2 = z (1 + t^2) + a b t (1 + z)$ and $(a+b t^2)^2 = z (1+ t^4) + a b t^2 (1+z)$. So we have $r^2 = s^2 = 0$ due to $t^8 = 1$ and $z^2 = 1$. Let $$\alpha = \frac{1}{2} (r (1-t) + s (1-t)^3).$$ Then $\alpha^2 = 0$. (We do not have to check $\alpha \neq 0$ when we have $e \alpha \not\in \zz[G]$ below.) We claim that $\alpha \in \zz[G]$. Note that $1 - t \equiv_2 1 + t$ and $(1+t^{2^{k}}) \equiv_2 (1+t)^{2^k}$ for every $k \in \nn$.  Thus, $r(1-t) \equiv_2 (a+b t) (1+t)^7 (1-z)$ and $s (1-t)^3 \equiv_2 (a + b t^2) (1+t)^7 (1-z)$. Moreover, $(a + b t) + (a + b t^2) \equiv_2 b t (1+t)$. Thus, $2 \alpha \equiv_2 b t (1+t)^8 (1-z) \equiv_2 b t (1+t^8) (1-z) \equiv_2 0$. In other words, $\alpha \in \zz[G]$.

Let $e = \widetilde{t^4}$ which is a central idempotent. We have $e r = r$ and $e s = 0$. Thus, $e \alpha = r (1-t) / 2$. Observe that $r (1-t) \equiv_2 (a+b t) (1+t^7) (1-z) \not\equiv_2 0$ since the coefficient of $a$ is odd. Hence, $e \alpha \not\in \zz[G]$, as desired.
\end{proof}

\vspace{12pt}
 
{\it Groups of type \textup{(BJ4)}}\\

Let $G$ be a group of type  (BJ4). The proof of \cite[Lemma 3.8]{LiuPas3} shows that $$G \simeq \langle x, y, z \mid x^9 = y^3 = 1, x y = y x, x^z = x y, y^z = x^{-3} y , z^3 = x^3 \rangle.$$ As in the proof of \cite[Lemma 2.1]{LiuPas2}, one shows that $G$ does not have ND. 

\begin{lem}[{\cite[Lemma 2.1]{LiuPas2}}]
\label{lemDecNilpBj4}
The group above does not have \textup{ND}.
\end{lem}

{\it Groups of type \textup{(BJ5)}}\\

Similarly, the proof of \cite[Lemma 2.15]{Liu} shows that groups of type (BJ5) do not have ND.  \\

\begin{lem}[{\cite[Lemma 2.15]{Liu}}]
\label{lemDecNilpBJ5}
Let $G = \langle a, b \mid a^8 = b^8 = 1, a^4 = b^4, a^b = a^{-1} \rangle$ be the group of type \textup{(BJ5)}. Then $G$ does not have \textup{ND}.
\end{lem}

\vspace{12pt}

{\it Groups of type \textup{(BJ6)}}\\

According to \cite[Example 3.5.7]{JesRio}, $\qq[Q_{16}] \simeq 4 \qq \oplus H \oplus M_2(\qq)$ where $\alpha = \xi_8 + \xi_8^{-1}$ and $H = \left( \frac{\alpha^2 - 4, -1}{\qq(\alpha)} \right)$ is the quaternion algebra over $\qq(\alpha)$ defined by $\alpha^2 - 4$ and $-1$. Since $\alpha$ is a real number and $\alpha^2 - 4 < 0$, it follows that $H$ is a division ring by \cite[Example 2.1.7(3)]{JesRio}. Thus, $\qq[Q_{16}]$ has only one matrix component. \\

\vspace{12pt}

{\it Groups of type \textup{(BJ7)}}\\

According to \cite[p. 123]{HalPasWil}, $\qq[Q_8 \Ydown D_8] \simeq 16 \qq \oplus M_2(\hh_{\qq})$ which has only one matrix component. \\

\vspace{12pt}

{\it Groups of type \textup{(BJ8)}}\\

As in Lemma~2.16 of \cite{Liu} one shows that the groups of type \textup{(BJ8)} do not have ND.

\begin{lem}[{\cite[Lemma 2.16]{Liu}}]
\label{lemDecNilpBJ8}
Let $G = \langle a, b, c \mid a^4 = b^4 = [a, b] = 1, c^2 = a^2, a^c = a b^2, b^c = b a^2 \rangle$ be the group of type \textup{(BJ8)}. Then $G$ does not have \textup{ND}.
\end{lem}

\vspace{12pt}

{\it Groups of type \textup{(BJ9)}}\\

We use the idea in (BJ8) to prove the following result.

\begin{lem}
\label{lemDecNilpBJ9}
The group $G = \langle a, b, c, d \mid a^4 = b^4 = [a, b] = 1, \, c^2 = a^2 b^2 , \, a^c = a^{-1}, \, b^c = a^2 b^{-1}, \, d^2 = a^2, \, a^d = a^{-1} b^2, \, b^d = b^{-1}, \, [c, d] = 1 \rangle$ of type \textup{(BJ9)} does not have \textup{ND}.
\end{lem}

\begin{proof}
Let $A = b$, $B = a$ and $C = c d$. Then $A^4 = B^4 = [A, B] = 1$ and $C^2 = c^2 d^2 = (a^2 b^2) a^2 = A^2$. Moreover, $A^C = b^{c d} = (b^{-1})^c = (b^3)^c = (b^c)^3 = a^2 b^{-3} = A B^2$ and $B^C = a^{c d} = (a^{-1} b^2)^c = a (a^2 b^{-1})^2 = A^2 B$. Let $H = \langle A, B, C \rangle$ and $G_0$ be the group of type (BJ8) with generators $a_0, b_0, c_0$. Then there is a group epimorphism $G_0 \to H$ given by $a_0 \mapsto A$, $b_0 \mapsto B$ and $c_0 \mapsto C$. As in the proof of \cite[Lemma~2.16]{Liu}, $(1-a_0^2 b_0^2) (1+a_0) (1+b_0) c_0$ is a nilpotent element of $\zz[G_0]$. It follows that $\alpha = (1 - A^2 B^2) (1 + A) (1 + B) C$ is a nilpotent element of $\zz[H] \subseteq \zz[G]$. Note that $A^2 = b^2 = a^2 c^2 = d^2 c^2$ which commutes with $c, d$ and, of course, $a$. Thus, $e = \widetilde{A^2}$ is a central idempotent in $\qq[G]$. We claim that $e\alpha \not\in \zz[G]$. It is equivalent to $e (\alpha C^{-1}) \not\in \zz[G]$. Now, $\widehat{A^2} (\alpha C^{-1})= \widehat{A^2} (1 - B^2) (1 + A) (1 + B) \equiv_2 \widehat{A} \widehat{B} = \widehat{b} \widehat{a}$. Since $\widehat{b} \widehat{a} = \widehat{\langle a, b \rangle} \not\in 2 \zz[G]$, it follows that $e (\alpha C^{-1}) \not\in \zz[G]$, as claimed.
\end{proof}


\vspace{12pt}
\subsection{Finite SSN groups that are nilpotent but not a $p$-group}$\;$
\label{secII42}

Now, we focus on nilpotent groups which are not $p$-groups. By Proposition~\ref{propSSN1}, we only have to focus on Hamiltonian groups.

\begin{lem}
\label{lemDecNilpHam}
Let $G$ be a Hamiltonian group such that $\qq[G]$ has nonzero nilpotent elements. Then $G$ has \textup{ND} if and only if $G = Q_8 \times C_p$ where $p$ is an odd prime such that $\ord_p(2)$ is even. In particular, $\qq[G]$ has only one matrix component which is isomorphic to $M_2(\qq(\xi_p))$.
\end{lem}

\begin{proof}
Write $G = Q_8 \times E \times A$, where $E$ is an elementary abelian $2$-group and $A$ is an abelian group of odd order. Assume that $G$ has ND. By assumption, $\qq[G]$ has nonzero nilpotent elements, so $\ord_{|A|}(2)$ is even. Note that $\qq[Q_8 \times A]$ also has nonzero nilpotent elements. By Lemma~\ref{lemDecNilp}, it follows that $E = 1$. Write $|A| = p_1^{n_1} \cdots p_k^{n_k}$ for distinct odd primes $p_1, \ldots, p_k$. Let $o_i = \ord_{p_i^{n_i}}(2)$. Then it is clear that $\ord_{|A|}(2) \mid o_1 o_2 \cdots o_k$. Thus, at least one of $o_i$ is even, say $o_1$. Let $S$ be the Sylow $p_1$-subgroup of $A$. Then $A = S \times T$ for some $p_1'$-group $T$. Since the algebra $\qq[Q_8 \times S]$ has nonzero nilpotent elements we get that   $T = 1$ (again by Lemma~\ref{lemDecNilp}). Thus, we can write $|A| = p^n$ for some odd prime $p$ such that $\ord_{p^n}(2)$ is even. Let $r = \ord_p(2)$. Then $2^r = 1 + p m$ for some $m \in \nn$. Observe that $2^{r p} = (1 + p m)^p \equiv 1 \pmod{p^2}$ and $2^{r p^{j}} \equiv 1 \pmod{p^{j+1}}$ for each $j \in \nn$ by induction. Thus, $\ord_{p^n}(2) \mid r p^{n-1}$. We can conclude that $r = \ord_p(2)$ is even. 

Let $C$ be a subgroup of $A$ of order $p$. Since $\qq[Q_8 \times C]$ has nonzero nilpotent elements, we can obtain that $A$ must be cyclic from Lemma~\ref{lemDecNilp}. We claim that $A = C \simeq C_p$ which is equivalent to $n=1$. Suppose $n \geq 2$ and let $A = \langle a \rangle$. The following argument imitates some steps in the proof of  \cite[Lemma~14]{HalPasWil}. Write $c = a^{p^{n-2}}$ and $Q_8 = \langle x, y \mid x^4 = 1, y^2 = x^2, y x y^{-1} = x^{-1} \rangle$. Since $\ord_p(2)$ is even and the order of $c^p$ is $p$, there are polynomials $r(X), s(X)$ in $\zz[X]$ such that $1 + r(c^p)^2 + s(c^p)^2 \in \zz \, \widehat{c^p}$ (see \cite[p. 153]{GiaSeh}). We have $$((1-c^p))^2 + ((1-c^p) r(c^p))^2 + ((1-c^p) s(c^p))^2 = 0.$$ Now, let $$w = \frac{1}{p} (1-x^2) [ (1-c)^{p^2- p - 1} (1-c^p) \alpha - (1-c)^{p-1} \widehat{c^p} \beta] ,$$ where $\alpha = x + r(c^p) y + s(c^p) x y$ and $\beta = x + r(c) y + s(c) x y$. Note that $(1 - c^p) \widehat{c^p} = 0$ and $((1-x^2) (1-c^p) \alpha)^2 = -2(1 - x^2) (1-c^p)^2 (1 + r(c^p)^2 + s(c^p)^2) = 0$. Moreover, $(1 - x^2) (1-c) \widehat{c^p} \beta$ can be regarded as a nilpotent element in $\qq[(Q_8 \times \langle c \rangle) / \langle c^p \rangle] \simeq \qq[Q_8 \times C]$. Thus, $((1 - x^2) (1-c) \widehat{c^p} \beta)^2 = 0$. So $w^2 = 0$. 

We illustrate that $w \in \zz[G]$. Note that $(1-c)^{p-1} \widehat{c^p} \equiv_p (1-c)^{p-1} (1+c+ \cdots + c^{p-1})^p \equiv_p (1-c^p)^{p-1} (1+c+ \cdots + c^{p-1}) \equiv_p (1-c)^{p^2-p} (1+c+ \cdots + c^{p-1}) \equiv_p (1-c)^{p^2-p-1} (1-c^p)$. Thus, $$p w \equiv_p (1-x^2) (1-c)^{p^2-p-1} (1-c^p) (\alpha - \beta) \equiv_p (1-x^2) (1-c)^{p^2-1} (\alpha - \beta).$$ Since $\alpha - \beta \in \zz[\langle c \rangle]$ and $\aug(\alpha - \beta) = 0$, it follows that $\alpha - \beta = (1-c) \gamma$ for some $\gamma \in \zz[\langle c \rangle]$. Hence, $p w \equiv_p (1-x^2) (1-c)^{p^2} \gamma \equiv_p (1-x) (1-c^{p^2}) \gamma \equiv_p 0$. In other words, $p w \in p \zz[G]$. We have $w \in \zz[G]$.

Let $e = \widetilde{c^p}$. Then $e$ is a central idempotent in $\qq[G]$. Moreover, $w e = - (1 - x^2) (1 - c)^{p-1} \beta \widetilde{c^p}$. Suppose that $w e \in \zz[G]$. Then $(1-c)^{p-1} \widetilde{c^p} \in \zz[\langle c \rangle]$. Thus, $0 \equiv_p (1-c)^{p-1} \widehat{c^p} \equiv_p (1-c)^{p^2-p-1} (1-c^p) \equiv_p (1-c)^{p^2-1}$. However, this leads to a contradiction because $(1-c)^{p^2-1} \not\equiv_p 0$. Hence, we must have $w e \not\in \zz[G]$, a contradiction. Thus, $n=1$. 

On the other hand, the rational group algebra $\qq[Q_8 \times C_p]$ is isomorphic to $4 \qq \oplus \hh_{\qq} \oplus 4 \qq(\xi_p) \oplus M_2(\qq(\xi_p)$ which has only one matrix component. This completes the proof. 
\end{proof}

\subsection{First main result: the nilpotent case}$\;$

Combining all results from (BJ1)-(BJ9) and Lemma~\ref{lemDecNilpHam}, we have a conclusion for nilpotent finite groups.

\begin{thmA}
\label{thmDecNilp2}
Let $G$ be a nilpotent \textup{SSN} group. Then $\qq[G]$ has only one matrix component if and only if $G$ is a group of the following types.
\begin{enumerate}
\item[(a)]
$G$ is a $p$-group and one of the following holds\textup{:}
\begin{enumerate}
\item[(i)]
$G = \langle a, b \mid a^{p^m} = b^{p^n} = 1, \, a^b = a^{1+p^{m-1}} \rangle = C_{p^m} \rtimes C_{p^n}$ where $p \geq 2$ is a prime, $m \geq 2$ and $n=1$, or $p = m = n = 2$ \textup{(}moreover, the matrix component is $M_p(\qq(\xi_{p^{m-1}}))$\textup{);}
\item[(ii)]
$G = G_0 \Ydown Z$, the central product of a nonabelian group $G_0$ of order $p^3$ with a cyclic group $Z$, where $G_0 \cap Z = \mathcal{Z}(G_0)$ and if $p=2$, then $|Z| > 2$
\textup{(}moreover,  the matrix component is $M_p(\qq(\xi_{|Z|}))$ for any $p$\textup{);}
\item[(iii)]
$G = Q_8 \times C_4$ \textup{(}moreover, the matrix component is $M_2(\qq(i))$\textup{);}
\item[(iv)]
$G = Q_{16}$ \textup{(}moreover, the matrix component is $M_2(\qq)$\textup{);}
\item[(v)]
$G = D_8 \Ydown Q_8$, an extraspecial $2$-group of order $2^5$ \textup{(}moreover,  the matrix component is $M_2(\hh (\qq))$\textup{)}.
\end{enumerate}
\item[(b)]
$G$ is not a $p$-group and $G = Q_8 \times C_p$ for some odd prime $p$ such that $\ord_p(2)$ is even \textup{(}moreover, the matrix component is $M_2(\qq(\xi_p))$\textup{)}.
\end{enumerate}
\end{thmA}

We know that rational group algebras of $p$-groups of type (BJ2), (BJ6) and (BJ7) have only one matrix component and $p$-groups of type (BJ3), (BJ4), (BJ5), (BJ8) and (BJ9) do not have ND. Combining these results and Lemma~\ref{lemDecNilpHam}, we can also conclude the following result.

\begin{cor}
\label{corDecNilpNilp}
Let $G$ be a nilpotent \textup{SSN} group that is not a group of type \textup{(BJ1)}. Assume that $\qq[G]$ has some nonzero nilpotent elements. If $G$ has \textup{ND}, then $\qq[G]$ has only one matrix component.
\end{cor}

We do not know whether groups of type (BJ1) have the above property.


\section{Nilpotent Decomposition for Non-nilpotent \textup{SSN} Groups}
\label{secII3}

We separate the study of the ND property for non-nilpotent SSN groups into three subsections. The first two subsections are based on the classifications of Theorem~\ref{thmSSN}(i) and (ii), respectively, for solvable groups that are not nilpotent. In the third subsection  we deal with the non-solvable groups and we show that there do not exist such groups that have both properties SSN and ND.
As in the previous section, we will use Theorem~\ref{thmMetabelian} to describe non-commutative  Wedderburn components  determined by strong Shoda pairs $(H,K)$,
and  with  $K \not\supseteq G'$  because of Lemma~\ref{lemSSPcomm}.


\subsection{Non-nilpotent \textup{SSN} groups that are  solvable  and have  \textup{ND}: faithful action}$\;$
\label{sec12-2}

As stated in Theorem~\ref{thmSSN}, there are two types of  non-nilpotent \textup{SSN} finite groups that are solvable. First we deal those listed in item (i) and we will prove that if they have property ND then  their rational group algebra has only one matrix  Wedderburn component. So, throughout this subsection, we assume that $G = P \rtimes Q$, where $P$ is an elementary abelian $p$-group for a prime $p$, $Q$ is a cyclic $p'$-group which acts faithfully on $P$, and every nontrivial subgroup of $Q$ acts irreducibly on $P$. 

Note that $P$ can be regarded as a vector space over $\ff_p$, the finite field of $p$ elements. Since $p \nmid |Q|$, every subgroup $H$ of $Q$ acts irreducibly on $P$ as an irreducible $\ff_p[H]$-module. It follows that the normalizer of every nontrivial proper subgroup of $P$ in $G$ is $P$. Thus one has $G' = P$ because $G/P \simeq Q$ is abelian and $G'$ is normal. Moreover, the faithful action implies that $P$ is the only maximal abelian subgroup containing $G'$. As a consequence, we have the following observation.

\begin{lem}
\label{lemSSN1}
Let $G = P \rtimes Q$ be as above. The following properties hold.
\begin{enumerate}
\item[(i)]
$N_G(P_1) = P$ for each nontrivial proper subgroup $P_1$ of $P$. 
\item[(ii)]
$G' = P$.
\item[(iii)]
$P$ is the maximal abelian subgroup containing $G'$.
\item[(iv)]
If $P$ is not cyclic, then $|Q|$ is odd.
\end{enumerate}
\end{lem}

\begin{proof}
The first three properties have been mentioned as above. We now prove (iv). Assume that $P$ is not cyclic. Suppose that $|Q|$ is even. Let $1 \neq g \in P$ and let $x \in Q$ be of order $2$. Then the element $g g^{x}$ is fixed by $x$. In particular, $x \in N_G(\langle g g^x \rangle)$. Then either $\langle g g^x \rangle = 1$ or $\langle g g^x \rangle = P$ by (i). Thus $g g^x = 1$ as by assumption $P$ is not cyclic. Hence, $g^x = g^{-1}$ and we obtain $x \in N_G(\langle g \rangle)$. This again is impossible by (i). Hence, $|Q|$ is odd.
\end{proof}

By Lemma~\ref{lemSSN1}(ii), $\qq[Q] \simeq \qq[G/P] = \qq[G/G']$ which is the direct sum of all commutative Wedderburn components of $\qq[G]$.

\begin{lem}
\label{lemDecNilp8}
Let $G = P \rtimes Q$ be as above. If $P$ is cyclic, then $\qq[G]$ has only one matrix component $\qq[G] (1 - \widetilde{G'})$. In particular, $G$ has \textup{ND}. 

Furthermore, the matrix component is isomorphic to $M_{|Q|}(F)$ where $F = \qq(\xi_p)^{\sigma}$ for some $\sigma \in \Aut(\qq(\xi_p))$ given by $\sigma(\xi_p) = \xi_p^i$. Here, the integer $i$ is such that  $g^h = g^i$ where $P = \langle g \rangle$ and $Q = \langle h \rangle$. Thus, $\qq[G] \simeq \qq[Q] \oplus M_{|Q|}(F)$.
\end{lem}

\begin{proof}
Suppose $P$ is cyclic, i.e. $P$ is of order $p$. Since $G$ is metabelian, we use Theorem~\ref{thmMetabelian} to find primitive central idempotents $e$ of $\qq[G]$ such that $\qq[G]e$ is non-commutative. Let $(H,K)$ be a pair of subgroups of $G$ as in Theorem~\ref{thmMetabelian} such that $K \not\supseteq G' = P$. Then $P \leq H$ and $H' \leq K \leq H$. It follows that $H' \leq G' \cap K = P \cap K = 1$. So $H$ is abelian and we obtain $H = P$ by Lemma~\ref{lemSSN1}(iii). Since $|P|=p$  we have $K = 1$. As a consequence, $e(G, P, 1)$ is the only primitive central idempotent $e$ such that $\qq[G] e$ is non-commutative. Moreover, $e(G, P, 1) = 1 - \widetilde{P} = 1 - \widetilde{G'}.$ Write $P = \langle g \rangle$, $Q = \langle h \rangle$ and let $i$ be such that $g^h = g^i$. From Proposition~\ref{propCyclicAlg} we get that $\qq[G] e(G, P, 1) \simeq \qq(\xi_p) * Q \simeq (\qq(\xi_p) / F, \sigma, 1) \simeq M_{|Q|}(F)$, where $Q \simeq \langle \sigma \rangle \subseteq \Aut(\qq(\xi_p))$, $F = \qq(\xi_p)^{\sigma}$ and $\sigma$ is  such that $\sigma(\xi_p) = \xi_p^i$. Hence, $\qq[G] \simeq \qq[Q] \oplus M_{|Q|}(F)$.
\end{proof}

The case when  $P$ is not cyclic is more complicated. First, we consider the Wedderburn decomposition of $\qq[G]$.

\begin{lem}
\label{lemSSN2}
Let $G = P \rtimes Q$ be as above. If $P$ is non-cyclic, then $$\qq[G] \simeq \qq[Q] \oplus v M_{|Q|}(\qq(\xi_p))$$ where $$v = \frac{|P| - 1}{|Q| (p-1)} \in \nn.$$
\end{lem}

\begin{proof}
The following argument is similar to that used in  the proof of Lemma~\ref{lemDecNilp8}. Since $G$ is metabelian, we use Theorem~\ref{thmMetabelian} to find non-commutative components $\qq[G]e$ for some primitive central idempotents $e$.  Let $(H,K)$ be a pair of subgroups of $G$ as in Theorem~\ref{thmMetabelian} such that $K \not\supseteq G' = P$. Note that $P \leq H$ and $H' \leq G' \cap K = P \cap K < P$. Suppose $H$ is nonabelian. We have $H' \neq 1$ so $P \lneq H$. There is an element $g x \in H$ for $g \in P$ and $1 \neq x \in Q$. Then $x \in H$. It follows that $x \in N_G(H')$. By Lemma~\ref{lemSSN1}(i), this can not happen. Hence, $H$ is abelian and, moreover, $H = P$ by Lemma~\ref{lemSSN1}(iii). Note that $P / K$ is cyclic if and only if $|P:K| = p$, because $P$ is elementary abelian. Consequently, we only need to focus on primitive central idempotents $e(G, P, K)$ with $K$ such that $|P:K| = p$. We remark here that $(P, K)$ is a strong Shoda pair of $G$. Thus, $$\qq[G] e(G, P, K) \simeq M_{|Q|}(\qq(\xi_p))$$ by Proposition~\ref{propSShodaPair}.

For convenience, we will give  an explicit form of $e(G, P, K)$. Let $K \leq P$ be such that $|P:K| = p$. Then $P / K \simeq C_p$. It follows that $\varepsilon(P,K) = \widetilde{K} - \widetilde{P}$. Clearly, $P \subseteq \cen_G(\varepsilon(P,K))$. Let $g x \in \cen_G(\varepsilon(P,K))$ for $g \in P$ and $x \in Q$. Then we obtain $\widetilde{K}^x = \widetilde{K}$. Again by Lemma~\ref{lemSSN1}(i), $x = 1$ and $\cen_G(\varepsilon(P,K)) = P$. So $Q$ is a right transversal of $\cen_G(\varepsilon(P,K))$ in $G$ and we have $$e_K = e(G, P, K) = \sum_{x \in Q} (\widetilde{K}^x - \widetilde{P}).$$ 

Next we determine how many distinct $e_K$ one can define this way. Note that $Q$ also acts on the set $S = \{K \mid K \leq P, \, |P:K| = p\}$. Every group in the orbit $K^Q$ leads to the same idempotent $e_K$. If $K_1^Q$ and $K_2^Q$ are distinct orbits, then $K_1^x \neq K_2^y$ for every $x, y \in Q$. Thus, $\widetilde{K_1}^x \widetilde{K_2}^y = \widetilde{P}$ and then $(\widetilde{K_1}^x - \widetilde{P}) (\widetilde{K_2}^y - \widetilde{P}) = 0$. It follows that $e_{K_1} e_{K_2} = 0$. In particular, we have $e_{K_1} \neq e_{K_2}$. Hence, the number of orbits $K^Q$ is precisely the number of distinct $e_K$. The number $|S|$ of subspaces of dimension $n-1$ in $P \simeq \ff_p^n$ is the number of subspaces of dimension $1$ and this is $(p^n - 1) / (p-1)$. Since each orbit $K^Q$ has size $|Q|$, we have $v = |S| / |Q| = (|P| - 1) / ((p-1) |Q|)$ distinct orbits. Therefore, we have $v$ distinct primitive central idempotents $e_1, \ldots, e_v$ such that $\qq[G] e_i \simeq M_{|Q|}(\qq(\xi_p))$ for each $i$. Finally, $\bigoplus_i \qq[G] e_i$ has dimension $v |Q|^2 (p-1) = |G| - |Q|$. Since $|Q| = |G/G'|$, it follows that these components $\qq[G] e_i$ for $i = 1, \ldots, v$ are all non-commutative components of $\qq[G]$.
\end{proof}

\begin{prop}
\label{propDecNilp7}
Let $G = P \rtimes Q$ be as above. Assume that $P$ is non-cyclic. If $G$ has \textup{ND}, then $\qq[G]$ has only one matrix component isomorphic to $M_{|Q|}(\qq(\xi_p))$. In this case, $|P| - 1 = |Q| (p-1)$ and $|Q|$ is a prime number.
\end{prop}

\begin{proof}
Let $G = P \rtimes Q$, where $P$ and $Q$ are as in Theorem~\ref{thmSSN}(i). Assume $P$ is non-cyclic and $G$ has ND. Since $Q$ is not normal in $G$, there are $s \in G$ and $y \in Q$ such that $s^{-1} y s \not\in Q$. Then $\alpha = (1-y) s \widehat{Q} = s (1 - y^{s^{-1}}) \widehat{Q}$ is a nonzero nilpotent element of $\zz[G]$. Write $y^{s^{-1}} = g_0 x$ for some $1 \neq g_0 \in P$ and $x \in Q$. Then $\alpha = s (1 - g_0) \widehat{Q}$. Let $K$ be a subgroup of $P$ such that $P / K$ is cyclic and let $e = e(G, P, K)$ be a primitive central idempotent of $\qq[G]$ as in the proof of Lemma~\ref{lemSSN2} (hence defining a non-commutative Wedderburn component). Because, by assumption, $G$ has ND, $e \alpha \in \zz[G]$. It follows that $e s^{-1} \alpha \in \zz[G]$ which implies $(1 - g_0) e \widehat{Q} \in \zz[G]$. Now, $\supp((1-g_0) e) \subseteq P$ so $(1-g_0) e$ must have integer coefficients. Thus, $$(1 - g_0) e = (1 - g_0) \sum_{x \in Q} (\widetilde{K}^x - \widetilde{P}) \in \zz[P].$$ Since $(1 - g_0) \widetilde{P} = 0$, we have $$(1 - g_0) \sum_{x \in Q} \widetilde{K}^x \in \zz[P].$$ If $g_0 \in K^x$ for some $x$, then $(1-g_0) \widetilde{K}^x = 0$. Computing the coefficient of $1$, it follows that $$\frac{|\{ x \in Q \mid g_0 \not\in K^x\}|}{|K|} \in \zz.$$ If $|\{ x \in Q \mid g_0 \not\in K^x\}| = 0$, then $g_0 \in K^x$ for all $x \in Q$. However, $\bigcap_{x \in Q} K^x$ is a proper subgroup of $P$ and it is normal in $G$. So $g_0 \in \bigcap_{x \in Q} K^x = 1$ by Lemma~\ref{lemSSN1}(i), a contradiction. Thus, $|\{ x \in Q \mid g_0 \not\in K^x\}|$ is a positive integer greater than or equal to $|K|$. In particular, $|Q| \geq |K|$.

According to Lemma~\ref{lemSSN2}, $|P| - 1 = v |Q| (p-1)$ for some $v \in \nn$. Now, we have $|P| - 1 \geq v |K| (p-1)$. Dividing by $|K|$ on both sides, we get $p - \frac{1}{|K|} \geq v (p-1)$. Since $0 < \frac{1}{|K|} < 1$, it follows that $p-1 \geq v (p-1)$. Hence, $v = 1$. So, by Lemma~\ref{lemSSN2}, $\qq[G]$ has only one matrix component, as desired.

Finally, we claim that $|Q|$ is prime. To see this, we first note that every nontrivial subgroup of $Q$ is not normal in $G$. Otherwise, there will be an abelian subgroup of $G$ properly containing $P = G'$ which contradicts to Lemma~\ref{lemSSN1}(iii). Thus, we can replace $\alpha$ by $\alpha = (1 - y) s \widehat{Q_1}$ in the above proof, where $Q_1$ is a subgroup of $Q$ such that $|Q_1|$ is the smallest prime divisor of $|Q|$. Then one also obtains $|Q_1| \geq |K|$. If $Q \neq Q_1$, then $|Q| \geq |Q_1|^2 \geq |K|^2 = |P|^2 / p^2 \geq |P|$ since $|P| \geq p^2$. It follows that $|Q| > |P|$ since $p \nmid |Q|$. However, this leads to $|P| - 1 = |Q| (p-1) > |P| (p-1) \geq |P|$, a contradiction. Thus, $Q = Q_1$.
\end{proof}

Combining Lemma~\ref{lemDecNilp8} and Proposition~\ref{propDecNilp7}, we can conclude the following result.

\begin{cor}
\label{corDecNilpNon-nilp}
Let $G = P \rtimes Q$ be a group in \textup{Theorem~\ref{thmSSN}(i)}. If $G$ has \textup{ND}, then $\qq[G]$ has only one matrix component.
\end{cor}

To find SSN groups $P \rtimes Q$ in Proposition~\ref{propDecNilp7}, one needs the action of $Q$ on $P$ to be faithful and irreducible, but in practice it suffices to have a nontrivial action and relations on the orders of $P$ and $Q$. Specifically, we have the following lemma.

\begin{lem}
\label{lemSSN3}
If $G = C_p^n \rtimes C_q$ is a group for some distinct primes $p$ and $q$ such that $p^n-1 = (p-1) q$ and $C_q$ acts nontrivially on $C_p^n$, then $G$ has \textup{SSN}. In particular, $\qq[G]$ has only one matrix component isomorphic to $M_q(\qq(\xi_p))$.
\end{lem}

\begin{proof}
For convenience, we write $P = C_p^n$ and $Q = C_q$. We first note that this nontrivial action must be faithful because every nontrivial element of $Q$ generates $Q$. Let $N$ be a proper subgroup of $P$ such that $Q$ acts on $N$. We claim that $N = 1$. Note that $P$ is an $\ff_p[Q]$-module and $N$ is a submodule. Since $\ff_p[Q]$ is semisimple, there exists a complement submodule $H$. Thus, $P = N \times H$ and $Q$ acts on $H$. If $N$ has an element $n$ such that $n^x \neq n$ for some $x \in Q$, then $n^x \neq n$ for all $1 \neq x \in Q$ and the orbit $n^{Q}$ consists of $q$ nontrivial elements. Thus, $N \setminus \{1\}$ contains $n^{Q}$. But $|N|-1 \leq p^{n-1} -1 < 1 + p + \cdots + p^{n-1} = q$, a contradiction. So $Q$ acts trivially on $N$. If $N \neq 1$, then $H$ is a proper subgroup of $P$. By the same argument again, we can conclude that $Q$ acts trivially on $H$. Thus, $Q$ acts trivially on $P$, a contradiction. Hence, indeed we have that $N = 1$. As a consequence, $Q$ acts irreducibly on $P$. Hence, $G$ has SSN by Theorem~\ref{thmSSN}(i). Finally, $\qq[G]$ has only one matrix component isomorphic to $M_q(\qq(\xi_p))$ since $v = (|P| - 1) / ((p-1) |Q|) = 1$ by Lemma~\ref{lemSSN2}.
\end{proof}

Combining Proposition~\ref{propDecNilp7} and Lemma~\ref{lemSSN3}, we have the following conclusion.

\begin{prop}
\label{propSSN3}
Let $G$ be as in \textup{Theorem~\ref{thmSSN}(i)} with non-cyclic $P$. Then $G$ has \textup{ND} if and only if $G = C_p^n \rtimes C_q$ for some primes $p, q$ such that $p^n-1 = (p-1) q$ and $C_q$ acts nontrivially on $C_p^n$. In this case, $\qq[G]$ has only one matrix component isomorphic and it is isomorphic to  $M_q(\qq(\xi_p))$.
\end{prop}

\begin{remk}
\label{remkDecNilp}
For primes $p, q$ and $n \geq 2$ such that $p^n - 1 = (p-1) q$, there always is a cyclic group $C_q$ acting nontrivially on $C_p^n$. To see this, we view $C_p^n$ as a $n$ dimensional vector space over $\ff_p$. Then it is well-known that $p^n - 1$ divides $|\GL_n(\ff_p)|$. Since $q$ is prime dividing $p^n-1$, we can always find an element in $\GL_n(\ff_p)$ of order $q$, as desired. 
If $p = 2$, then $q$ is a {\it Mersenne prime} and for each such prime $2^n-1$ we get an SSN group, e.g. $C_2^2 \rtimes C_3 \simeq A_4$, $C_2^3 \rtimes C_7$, $C_2^5 \rtimes C_{31}$. If $p$ is odd, then so are $q$ and $n$ since $q = 1 + p + \cdots + p^{n-1} \neq 2$ and $q\equiv n \pmod{2}$. For instance, $C_3^3 \rtimes C_{13}$, $C_5^3 \rtimes C_{31}$ and $C_7^5 \rtimes C_{2801}$ are SSN groups. Primes $p$ and $q$ satisfying $p^n-1 = (p-1) q$ for some $n$ are called {\it repunit primes}. For more information, one can refer to \cite{Dub}.
\end{remk}


\subsection{Non-nilpotent \textup{SSN}  groups that are  solvable  and have  \textup{ND}: action not faithful}
\label{sec12-4}

In this subsection we deal with  SSN groups $G$ listed in  Theorem~\ref{thmSSN}(ii). Write the presentation of $G$ as follows: 
$$G = \langle x, y\mid x^p = y^{q^k} = 1, y x y^{-1} = x^{r_0} \rangle ,$$ 
for some distinct primes $p, q$, integers $k, r_0 \geq 2$ and $\langle y \rangle$ acts non-faithfully on $\langle x \rangle$. Then there is a positive integer $1 < k_0 < k$ such that $y^{q^{k_0}}$ fixes $x$ and $y^{q^l}$ does not fix $x$ for every $0 \leq l < k_0$. It is easy to check that $\ord_p(r_0) = q^{k_0}$. Note that $G' = \langle x \rangle$ and $M = \langle x, y^{q^{k_0}} \rangle$ is the maximal abelian subgroup of $G$ containing $G'$. Moreover, $M = \langle x y^{q^{k_0}} \rangle$ is cyclic of order $p q^{k-k_0}$. If $K$ is a normal subgroup of $M$ such that $(M,K)$ is a strong Shoda pair of $G$, then $\qq[G]e(G,M,K)$ is commutative if and only if $K \supseteq G'$ by Lemma~\ref{lemSSPcomm}. Moreover, if $K$ contains $x^i (y^{q^{k_0}})^j$ for some $i, j$ with $p \nmid i$, then $K$ contains an element $\langle x \rangle = G'$. Hence, by Theorem~\ref{thmMetabelian}, the elements $$e_j = e(G, M, \langle y^{q^j} \rangle),$$ for $k_0 \leq j \leq k$, are the primitive central idempotents corresponding to non-commutative Wedderburn components. More precisely, $$e_{k_0} = (1 - \widetilde{x}) \widetilde{y^{q^{k_0}}} \quad \text{and} \quad e_j = (1 - \widetilde{x}) (\widetilde{y^{q^j}} - \widetilde{y^{q^{j-1}}})$$ for $k_0 < j \leq k$ and $$\qq[G] \simeq \qq[G] \widetilde{G'} \oplus \qq[G] e_{k_0} \oplus \cdots \oplus \qq[G] e_k.$$

Now, we fix $p, q, r_0$ and denote $G_k = G$. Then $G_k / \langle y^{q^j} \rangle \simeq G_{j}$ for $k_0 \leq j \leq k$. In particular, $\qq[G_k] \widetilde{y^{q^{j}}} \simeq \qq[G_k / \langle y^{q^{j}} \rangle] \simeq \qq[G_{j}]$. Denote $$C_{k_0} = \qq[G_k] \widetilde{x} \widetilde{y^{q^{k_0}}} \quad \text{and} \quad C_j = \qq[G_k] \widetilde{x} (\widetilde{y^{q^j}} - \widetilde{y^{q^{j-1}}})$$ for $k_0 < j \leq k$ which are commutative algebras. Then $$\qq[G_k] \simeq \qq[G_{k-1}] \oplus C_k \oplus \qq[G_k] e_k$$ and $$\qq[G_k] \simeq \qq[G_{k-2}] \oplus C_{k-1} \oplus C_k \oplus \qq[G_k] e_{k-1} \oplus \qq[G_k] e_k.$$ If we denote $\overline{x}$, $\overline{y}$ to be generators of $G_{j}$ for $j < k$, then $$\qq[G_{k-1}] \simeq \qq[G_{k-2}] \oplus \overline{C}_{k-1} \oplus \qq[G_{k-1}] \overline{e}_{k-1}$$ so we have $$\qq[G_k] \simeq \qq[G_{k-2}] \oplus \overline{C}_{k-1} \oplus \qq[G_{k-1}] \overline{e}_{k-1} \oplus C_k \oplus \qq[G_k] e_k.$$ By the uniqueness of Wedderburn decomposition, we have $$C_{k-1} \simeq \overline{C}_{k-1} \quad \text{and} \quad \qq[G_k] e_{k-1} \simeq \qq[G_{k-1}] \overline{e}_{k-1}.$$ Continuing this process, we can obtain that $$C_j \simeq \overline{C}_j \quad \text{and} \quad \qq[G_k] e_j \simeq \qq[G_j] \overline{e}_j$$ for $k_0 \leq j < k$. Thus, $$\qq[G_k] \simeq \qq[G_k] \widetilde{G_k'} \oplus \bigoplus_{k_0}^{k-1} \qq[G_j] \overline{e}_j \oplus \qq[G_{k}] {e}_{k}.$$ Now, $G_{k_0}$ is a group as in Theorem~\ref{thmSSN}(i). By Lemma~\ref{lemDecNilp8}, $$\qq[G_{k_0}] \overline{e}_{k_0} \simeq M_{q^{k_0}}(\qq(\xi_p)^{\sigma})$$ for $\sigma \in \Aut(\qq(\xi_p))$ given by $\sigma(\xi_p) = \xi_p^{r_0}$. If all components $\qq[G_j] \overline{e}_j$, $k_0 < j < k$, and $\qq[G_k] e_k$ are division rings, then $\qq[G_k]$ has only one matrix component, and vice versa. In this case, $G_k$ has ND. It suffices to consider only $\qq[G_k] e_k$ by the reduction process above. To check that $\qq[G] e_k$ is a division ring or not, we will use S.A. Amitsur's result on \cite{Ami}. 

\begin{thm}[{\cite[Theorem 3]{Ami}}]
\label{thmAmi}
Let $m, r$ be two relatively prime integers and put $s = \gcd(r-1, m)$, $t = m / s$ and $n = \ord_m(r)$. 
Consider the group  
$$G_{m,r} = \langle A, B \mid A^m = 1, B^n = A^t, B A B^{-1} = A^r \rangle$$
and let 
$$\mathfrak{A}_{m,r} = (\qq(\xi_m), \sigma_r, \xi_s),$$ 
a cyclic algebra where $\sigma_r(\xi_m) = \xi_m^r$. Then $G_{m, r}$ can be embedded in a division ring $D$ if and only if $\mathfrak{A}_{m,r}$ is a division algebra\textup{;} and then $\{\sum_i a_i g_i \in D \mid a_i \in \qq, g_i \in G \} \simeq \mathfrak{A}_{m,r}$ where the isomorphism is obtained by the correspondence\textup{:} $A \leftrightarrow \xi_m$, $B \leftrightarrow \sigma_r$.
\end{thm}

We list conditions \cite[(3C), (3D)]{Ami} below for the next theorem where $m, n, s, t$ are defined as in the previous theorem:
\begin{enumerate}
\item[(3C)]
\textup{$\gcd(n, t) = \gcd(s, t) = 1$;}
\item[(3D)]
\textup{$n = 2 n'$, $m = 2^{\alpha} m'$, $s = 2 s'$, where $\alpha \geq 2$, $m'$, $s'$, $n'$ are odd numbers; $\gcd(n, t) = \gcd(s, t) = 2$, and $r \equiv -1 \pmod{2^{\alpha}}$.}
\end{enumerate}
We define the following notations. Let $p$ be a fixed prime dividing $m$. We put: $$\alpha_p = v_p(m), \quad n_p = \ord_{m p^{-\alpha_p}}(r), \quad \delta_p = \ord_{m p^{-\alpha_p}}(p)$$ and $\mu_p$ is the minimal integer satisfying $r^{\mu_p} \equiv p^{\mu'} \pmod{m p^{-\alpha_p}}$ for some integer $\mu'$ where $v_p$ is the $p$-adic valuation, in other words, $v_p(m)$ is the largest nonnegative integer such that $p^{v_p(m)}$ divides $m$. Let $$\delta' = \mu_p \delta_p / n_p.$$

\begin{thm}[{\cite[Theorem 4]{Ami}}]
\label{thmAmi2}
Let $m, n, r, s, t$ be as in \textup{Theorem~\ref{thmAmi}}. The cyclic algebra $\mathfrak{A}_{m,r} = (\qq(\xi_m), \sigma_r, \xi_s)$ is a division algebra if and only if either \textup{(3C)} or \textup{(3D)} holds and one of the following holds\textup{:}
\begin{enumerate}
\item[(1)]
$n = s = 2$ and $r \equiv -1 \pmod{m}$.
\item[(2)]
For every $q \mid n$ there exists a prime $p \mid m$ such that $q \nmid n_p$ and that either 
\begin{enumerate}
\item[(a)]
$p \neq 2$, and $\gcd(q, (p^{\delta'} - 1) / s) = 1$ or,
\item[(b)]
$p = q = 2$, \textup{(3D)} holds, and $m / 4 \equiv \delta' \equiv 1 \pmod{2}$.
\end{enumerate}
\end{enumerate}
\end{thm}

To use Amitsur's result, we have to rewrite the presentation of $G_k$ to fit all notations in Theorem~\ref{thmAmi}. Since $\gcd(p, q^{k-k_0}) = 1$, choose $p'$ such that $p p' \equiv 1 \pmod{q^{k-k_0}}$. Let $$A = x y^{p' q^{k_0}}, \quad B = y, \quad m = p q^{k-k_0} \quad \text{and} \quad n = q^{k_0}.$$ Then $|A| = m$. By Chinese Remainder Theorem, let $r \in \zz$ be such that $$r \equiv r_0 \pmod{p} \quad \text{and} \quad r \equiv 1 \pmod{q^{k-k_0}}.$$ We have $$B A B^{-1} = y x y^{p' q^{k_0}} y^{-1} = x^{r_0} y^{p' q^{k_0}} = x^r y^{p' q^{k_0} r} = A^r.$$ Here, $\gcd(m, r) = 1$ since $p \nmid r_0$. Let $$s = \gcd(r-1, m), \quad t = m / s \quad \text{and} \quad n = \ord_m(r).$$ We compute these three numbers. Since $p \nmid r_0-1$, we have $p \nmid r-1$. Moreover, $q^{k-k_0} \mid r-1$. We obtain $$s = q^{k-k_0} \quad \text{and} \quad t = p.$$ Note that $r^{q^{k_0}} \equiv r_0^{q^{k_0}} \equiv 1 \pmod{p}$ since $y^{q^{k_0}}$ fixes $x$. Thus, $r^{q^{k_0}} \equiv 1 \pmod{m}$ and $n \mid q^{k_0}$. In other words, $n$ is a power of $q$. On the other hand, $r_0^n \equiv r^n \equiv 1 \pmod{p}$. It follows that $y^{n}$ fixes $x$. By the definition of $k_0$, we have $$n = q^{k_0}.$$ Therefore, $$B^n = y^{q^{k_0}} = A^p = A^t.$$ Now, if $H$ is a group generated by the above generators $A, B$ and adding the stated relations, then $|H| = m n = p q^k$. There is an epimorphism from $G_k$ into $H$ since $x$ and $y$ can be rewritten by $A$ and $B$. It follows that $G_k \simeq H$. In other words, $$G_k = \langle A, B \mid A^m = 1, B^n = A^t, B A B^{-1} = A^r \rangle.$$ In particular, $$G_k' = \langle A^s \rangle \quad \text{and} \quad \mathcal{Z}(G_k) = \langle A^t \rangle.$$ Observe that $\gcd(n, t) = \gcd(s, t) = 1$. So $G_k$ satisfies the condition (3C). 

Recall that $e_k = e(G_k,M,1)$ where $M = \langle x, y^{q^{k_0}} \rangle = \langle x y^{q^{k_0}} \rangle$ is the maximal abelian subgroup of $G_k$ containing $G_k'$. 

\begin{lem}
\label{lemIsomDiv}
$\qq[G_k] e_k \simeq  (\qq(\xi_m), \sigma_r, \xi_s) = \mathfrak{A}_{m,r}$ for $k \geq 2$.
\end{lem}

\begin{proof}
We follow all notations as above. First of all, we note that $M = \langle A \rangle$ since $|A| = m = |M|$. Moreover, $A^y = A^B = A^r$. Observe that $|G_k/M| = q^{k_0} = n$. For $0 \leq i < n$, we fix preimages of $\overline{y}^i \in G_k/M$ by $y^i$ under the canonical homomorphism $G \to G/M$. For $0 \leq i, j < n$, we have $y^i y^j = y^n y^{i+j-n}$ where $0 \leq i+j - n < n$. Moreover, $y^n = B^n = A^t = A^p$. By Proposition~\ref{propSShodaPair}, we obtain that $\qq[G_k]e_k \simeq (\qq(\xi_m), \sigma_r, \xi_m^{p})$. Now, $\xi_m^{p} = \xi_s$ and the result follows.
\end{proof}

As we mentioned, our group $G_k$ satisfies (3C). We check conditions in Theorem~\ref{thmAmi2}(1). If $n = s = 2$, then $q = 2$, $k_0 = 1$ and $k = 2$. Furthermore, $m = 2 p$ and $\ord_p(r_0) = 2$. Thus, we can assume $r_0 = p-1$ and $r \equiv r_0 \equiv -1 \pmod{p}$. It follows that $r \equiv -1 \pmod{m}$ since $q^{k-k_0} = 2$ and $r \equiv 1 \equiv -1 \pmod{q^{k-k_0}}$. By Theorem~\ref{thmAmi2}, $\qq[G_k] e_k = \qq[G_2] e_2$ is a division ring. In this case, $$G_2 \simeq Q_{4p} = \langle x, y \mid x^p = y^4 = 1, \, y x y^{-1} = x^{-1} \rangle,$$ the {\it generalized quaternion group} of order $4p$ for an odd prime $p$.

For the case (2) in Theorem~\ref{thmAmi2}, we observe that $n = q^{k_0}$, $m = p q^{k - k_0}$. In particular, $m$ has two prime divisors $p$ and $q$. For $p$, we have $\alpha_p = v_p(m) = 1$ and $n_p = \ord_{m p^{-\alpha_p}}(r) = \ord_{q^{k-k_0}}(r) = 1$; for $q$, we have $\alpha_q = v_q(m) = k - k_0$ and $n_q = \ord_{m q^{-\alpha_q}} (r) = \ord_p(r) = q^{k_0}$. Note that $p > 2$ since $G_k$ is nonabelian. Thus, the only possibility is the case (2)(a) since $p \neq q$. Hence, we only have to check when $$\gcd(q, (p^{\delta'} - 1) / s) = 1$$ happens assuming we are not in the situation $n = s = 2$.

We compute $\delta'$. Since $\alpha_p = 1$ and $r \equiv 1 \pmod{q^{k-k_0}}$, we have $\mu_p = 1$. Moreover, $n_p = 1$. Then we get $$\delta' = \delta_p = \ord_{q^{k-k_0}}(p).$$ Recall that $v_q$ denotes the $q$-adic valuation. It follows that $\gcd(q, (p^{\delta'} - 1) / s) = 1$ if and only if $v_q((p^{\delta'} - 1) / s) = 0$ if and only if $$v_q(p^{\ord_{q^{k-k_0}}(p)} - 1) = v_q(s) = k - k_0.$$ Hence, we have the following

\begin{lem}
\label{lemCriterionDiv}
Assume that $n = s = 2$ does not hold. Then $\qq[G_k] e_k$ is a division ring if and only if $v_q(p^{\ord_{q^{k-k_0}}(p)} - 1) = k - k_0$.
\end{lem}

\begin{lem}
\label{lemDecNilp9}
Assume that $\qq[G_k]$ has only one matrix component for $k \geq 2$. Then $y^q$ fixes $x$ and the following holds.
\begin{enumerate}
\item[(i)]
If $q \neq 2$, then $v_q(p-1) = 1$.
\item[(ii)]
If $q = 2$, then either $k = 2$ and $G_2 \simeq Q_{4p}$, or $k > 2$ and $p \equiv 5 \pmod{8}$.
\end{enumerate}
\end{lem}

\begin{proof}
Note that $\ord_p(r_0) = q^{k_0}$ and $r_0^{p-1} \equiv 1 \pmod{p}$. Thus, $q^{k_0} \mid p-1$. In other words, $v_q(p-1) \geq k_0$. By assumption, $\qq[G_j] e_j$ are division rings for $k_0+1 \leq j \leq k$. We first assume $q \neq 2$. Then, by Lemma~\ref{lemCriterionDiv}, we have $v_q(p^{\delta'_j} - 1) = j - k_0$ where $\delta'_j = \ord_{q^{j-k_0}}(p)$ for each $j$. When $j = k_0 + 1$, we obtain $\delta'_j = 1$ and $v_q(p - 1) = 1 \geq k_0$. In particular, $k_0 = 1$ and $y^q$ fixes $x$.

Assume $q = 2$. Recall that $n = 2^{k_0}$ and $s = 2^{j-k_0}$ for the group $G_j$. When $j = k_0+1$, we obtain $s = 2$. If $k_0 > 1$, then $n \neq s = 2$ and we have $v_2(p^{\delta'_j} - 1) = j - k_0 = 1$ by Lemma~\ref{lemCriterionDiv}. However, we also have $\delta'_j = 1$ and $v_2(p^{\delta'_j} - 1) = v_2(p-1) \geq k_0 > 1$, a contradiction. Thus, $k_0 = 1$ and $n = 2$. In this case, $y^{q}$ fixes $x$.

If $k = 2$, then we can obtain that $G_2 \simeq Q_{4p}$. Assume $k \geq 3$. When $j = k_0 + 2$, we have $s = 2^2 \neq n = 2$. So $v_2(p^{\delta'_j} - 1) = j - k_0 = 2$ by Lemma~\ref{lemCriterionDiv} again. On the other hand, $\delta'_j = \ord_{2^2}(p)$ which divides $2$. If $\delta'_j = 2$, we have $p \equiv 3 \pmod{4}$. Then $p^2 \equiv 1 \pmod{8}$ and $v_2(p^{\delta'_j} - 1) = v_2(p^2 - 1) \geq 3$, a contradiction. Thus, $\delta'_j = 1$ and $v_2(p-1) = 2$. It follows that $p \equiv 5 \pmod{8}$. 
\end{proof}

We are going to prove that the converse of Lemma~\ref{lemDecNilp9} is also true. Before we prove it, we give an elementary computation.

\begin{lem}
\label{lempadic}
Let $a, b, c \in \nn$ be such that $c = 1 + q^a b$ and $q \nmid b$, and $a > 1$ if $q = 2$. Then $\ord_{q^{a+d}}(c) = q^d$ and $v_q(c^{q^d} - 1) = a + d$ for every $d \in \zz_{\geq 0}$. 
\end{lem}

\begin{proof}
Computing $c^q$, we have $c^q =  (1 + q^a b)^q = \sum_{i=0}^q \binom{q}{i} (q^a b)^i$. Note that $v_q(\binom{q}{i}) = 1$ for $i \neq 0, q$. Because of the assumptions that $q \nmid b$ and $a > 1$ if $q = 2$, it is easily seen that $c^q = 1 + q^{a+1} b_1$ for some $b_1 \in \nn$ and $q \nmid b_1$. Continuing this process, we  obtain $c^{q^d} = 1 + q^{a+d} b_d$ for some $b_d \in \nn $ and $q \nmid b_d$, for every $d \in \zz_{\geq 0}$ where $b_0 = b$. Thus, $v_q(c^{q^d} - 1) = a+d$. Moreover, $\ord_{q^{a+d}}(c) \mid q^d$. If $d = 0$, then $\ord_{q^{a+0}}(c) = 1$. Assume $d > 0$. Note that $c^{q^{d-1}} = 1 + q^{a+d-1} b_{d-1} \not\equiv 1 \pmod{q^{a+d}}$ since $q \nmid b_{d-1}$. It follows that $\ord_{q^{a+d}}(c) = q^d$, as desired. 
\end{proof}

\begin{prop}
\label{propDecNilp2}
Let $k \geq 2$. Then $\qq[G_k]$ has only one matrix component if and only if $y^q$ fixes $x$ and the following holds.
\begin{enumerate}
\item[(i)]
If $q \neq 2$, then $v_q(p-1) = 1$.
\item[(ii)]
If $q = 2$, then either $k = 2$ and $G_2 \simeq Q_{4p}$, or $k > 2$ and $p \equiv 5 \pmod{8}$.
\end{enumerate}
In particular, the matrix component is isomorphic to $M_{q}(\qq(\xi_p)^{\sigma})$ for $\sigma \in \Aut(\qq(\xi_p))$ given by $\sigma(\xi_p) = \xi_p^{r_0}$.
\end{prop}

\begin{proof}
One direction is actually Lemma~\ref{lemDecNilp9} and we prove the opposite direction. By assumption, $k_0 = 1$. Let $q \neq 2$ and $v_q(p-1) = 1$. Then $p = 1 + q b$ for some $b \in \nn$ and $q \nmid b$. We can conclude that $\delta'_j = \ord_{q^{j - k_0}}(p) = q^{j - k_0 - 1}$ and $v_q(p^{\delta'_j} - 1) = j - k_0$ for $k_0 + 1 \leq j \leq k$ by Lemma~\ref{lempadic}. Thus, by Lemma~\ref{lemCriterionDiv}, $\qq[G_j] e_j$ is a division ring for $k_0 + 1 \leq j \leq k$ and $\qq[G_k]$ has only one matrix component.

Assume $q = 2$. If $k = 2$, then $G_2 \simeq Q_{4 p}$. In this case, $\qq[Q_{4p}]$ has only one matrix component isomorphic to $M_2(\qq(\xi_p + \xi_p^{-1}))$ by \cite[Example~3.5.7]{JesRio}. Assume $k \geq 3$ and $p \equiv 5 \pmod{8}$. Then $p = 1 + 2^2 b$ for some odd integer $b$. Thus, we also have that $\delta'_j = \ord_{2^{j - k_0}}(p) = 2^{j - k_0 - 2}$ and $v_2(p^{\delta'_j} - 1) = j - k_0$ for $k_0 + 2 \leq j \leq k$ by Lemma~\ref{lempadic}. So $\qq[G_j] e_j$ is a division ring for every $k_0 + 2 = 3 \leq j \leq k$ by Lemma~\ref{lemCriterionDiv}. 

Finally, since $k_0 = 1$ and $G_1$ is the group in Theorem~\ref{thmSSN}, the matrix component comes from $\qq[G_1] e_1$ which is isomorphic to the matrix ring described in Lemma~\ref{lemDecNilp8}. We complete the proof.
\end{proof}

\begin{remk}
\label{remkSSNiiND}
Let $G = C_p \rtimes C_{q^k}$ be  a group as  in Theorem~\ref{thmSSN}(ii) with distinct primes $p, q$ and $k \geq 2$. If $p \equiv 3 \pmod{4}$, $q = 2$ and $k \geq 3$, then Hales and Passi proved that $G$ does not have ND by finding a tricky nilpotent element in \cite[Theorem 4.1]{HalPas2}. However, we still do not know in general whether $G$ has ND or not if $\qq[G]$ has more than one matrix component.
\end{remk}

Recall that if $\qq[G_k] e_k$ is a division ring , then $G_k$ is embedded in $\mathfrak{A}_{m,r}$ by Lemma~\ref{lemIsomDiv} and Theorem~\ref{thmAmi}. Note that $v_3 (7 - 1) = 1$ so we have the following infinite number of groups of odd order which can be embedded in division rings. See also \cite[Theorem~6]{Ami}.

\begin{cor}[{\cite[Proposition~4.7]{LiuPas3}}]
\label{cor73k}
The nonabelian group $C_7 \rtimes C_{3^k} = \langle x, y \mid x^7 = y^{3^k} = 1, \, y x y^{-1} = x^2 \rangle$ for $k \geq 2$ are all embeddable in division rings.
\end{cor}

According to \cite[Theorem~6]{Ami}, the group $C_7 \rtimes C_9$ above is a group of smallest possible odd order embedded in a non-commutative division ring. Note also that $v_3(13-1) = 1$. Then we have another infinite number of groups $C_{13} \rtimes C_{3^k} = \langle x, y \mid x^{13} = y^{3^k} = 1, \, y x y^{-1} = x^3 \rangle$ of odd order for $k \geq 2$ embedded in non-commutative division rings. When $k = 2$, the group $C_{13} \rtimes C_{9}$ is the group of next smallest odd order embedded in a division ring, see \cite[(3.7)~Theorem]{Lam2}.


\subsection{Non-solvable \textup{SSN} Groups and ND}$\;$
\label{sec12-3}

According to Theorem~\ref{thmSSN2}, $A_5$ is the unique non-solvable finite SSN group. Note that $(A_4, K)$ is a Shoda pair of the group $A_5$ where $A_4 = \langle (1,2,3), (1,2) (3,4) \rangle$ and $K = \langle (1,2) (3,4), (1,3) (2,4) \rangle$. Let 
$$e = \frac{1}{2} e(A_5, A_4, K) = \frac{1}{2} \sum_{i=0}^4 \varepsilon(A_4, K)^{a^i},$$ 
where $\varepsilon(A_4, K) = \widetilde{K} - \widetilde{A_4}$, 
$a = (1,2,3,4,5)$ and $\varepsilon(A_4, K)^{a^i} = a^i \varepsilon(A_4, K) {a^{-i}}$. 
Then $e$ is a primitive central idempotent of $\qq[A_5]$ (see \cite[Example 3.4.5]{JesRio}). Let 
$$\alpha = \widehat{a} (1,2) (3,4) (1 - a),$$ 
a nonzero nilpotent element of $\zz[A_5]$. We claim that $\alpha e \not\in \zz[A_5]$. We compute the coefficient of $(1,2) (3,4)$ in $\alpha e$. For convenience, let $b = (1,2) (3,4)$ and $c = (1,2,3)$. Compute $\alpha \varepsilon(A_4,K)^{a^i}$ for $i = 0, 1, 2, 3, 4$. Note that $\widehat{a} b \varepsilon(A_4,K)^{a^i}$ always has support $b$ with the coefficient $\frac{1}{4} - \frac{1}{12}$ for each $i$. Moreover, $\widehat{a} b (-a) \varepsilon(A_4,K)^{a^i}$ has $b$ in its support, see the following table.
$$\begin{array}{|c|c|c|c|c|c|} \hline
i & 0 & 1 & 2 & 3 & 4 \\ \hline
b = & a^3 b a b c^2 & a b a (b c^2 b)^a & a^2 b a (c b c^2)^{a^2} & a b a (b c)^{a^3} & a^3 b a c^{a^4} \\ \hline
\text{coeff.} & \frac{1}{12} & \frac{1}{12} & - (\frac{1}{4} - \frac{1}{12}) & \frac{1}{12} & \frac{1}{12} \\ \hline
\end{array}$$
Hence, $\alpha e$ has $b$ in its support and it has  coefficient $$\frac{1}{2} \left(\frac{1}{4} + \frac{1}{4} + 0 + \frac{1}{4} + \frac{1}{4}\right) = \frac{1}{2}.$$ (This can also be checked by \cite{GAP}.) Thus, $\alpha e \not\in \zz[A_5]$. In other words, $A_5$ does not have ND.

It follows from Theorem~\ref{thmSSN2} and the argument above that

\begin{cor}
\label{corDecNilp}
If a group has both \textup{ND} and \textup{SSN}, then it is solvable.
\end{cor}

\subsection{Second main result: the non-nilpotent case}$\;$

As a consequence of  Theorem~\ref{thmSSN}, Lemma~\ref{lemDecNilp8}, Proposition~\ref{propSSN3}, Proposition~\ref{propDecNilp2} and Corollary~\ref{corDecNilp}, we  get the following result.

\begin{thmA}
\label{thmDecNilp}
Let $G$ be a non-nilpotent \textup{SSN} finite group. Then $\qq[G]$ has only one matrix component if and only if $G$ is one of the following\textup{:}
\begin{enumerate}
\item[(i)]
$G = \langle x \rangle_p \rtimes \langle y \rangle_n$ where $\langle y \rangle_n$ acts faithfully on $\langle x \rangle_p$\textup{;}
\item[(ii)]
$G = C_p^n \rtimes C_q$ where $n \geq 2$, $p^n - 1 = (p-1) q$ and $C_q$ acts nontrivially on $C_p^n$\textup{;}
\item[(iii)]
$G = \langle x \rangle_p \rtimes \langle y \rangle_{q^k}$ where $k \geq 2$, $y$ acts nontrivially on $x$, $y^q$ acts trivially on $x$ and one of the following holds\textup{:}
\begin{enumerate}
\item[(a)]
$q \neq 2$ and $v_q(p-1) = 1$\textup{;}
\item[(b)]
$q = 2$ and either $k = 2$, or $k > 2$ and $p \equiv 5 \pmod{8}$.
\end{enumerate}
\end{enumerate}
Here $p, q$ are distinct primes and $n, k \in \nn$.

In particular, for groups of type \textup{(i)} and \textup{(iii)}, the matrix components are isomorphic to $M_n(\qq(\xi_p)^{\sigma})$ and $M_q(\qq(\xi_p)^{\sigma})$ where $\sigma(\xi_p) = \xi_p^i$ if $x^y = x^i$; for the group of type \textup{(ii)}, it is $M_q(\qq(\xi_p))$. Moreover, rational group algebras of type \textup{(i)} and \textup{(ii)} have only one non-commutative component\textup{;} for type \textup{(iii)}, it has $k-1$ non-commutative components which are division algebras.
\end{thmA}

We present a conjecture below based on Corollary~\ref{corDecNilpNilp} and Corollary~\ref{corDecNilpNon-nilp}.

\begin{conjecture}
\label{conjND}
Let $G$ be a finite group such that $\qq[G]$ has some nonzero nilpotent elements. If $G$ has ND, then $\qq[G]$ has only one matrix component.
\end{conjecture}

\section{Concluding Remark}

Let $G$ be a finite group. Denote by $\{e_i\}_{i \in I}$ the collection of primitive central idempotents of $\qq[G]$. If $\alpha \in \zz[G]$ is nilpotent, then $\alpha = \sum_i \alpha_i$ where $\alpha_i = \alpha e_i$. It follows that $\alpha_i \alpha_j = \alpha_j \alpha_i = 0$ for $i \neq j$. Thus, $1 + \alpha = \prod_{i \in I} (1 + \alpha_i)$. In other words, a unipotent unit of $\zz[G]$ is a product of unipotent units of $\qq[G]$ determined by nilpotents in the Wedderburn components of $\qq[G]$. For each $i$, there exists $n_i \in \nn$ such that $(1 + \alpha_i)^{n_i} \in \zz[G]$ (for example if $\alpha_i^{k+1} = 0$ and $m \alpha_i \in \zz[G]$ then we can choose $n_i = k! m^k$ such that $\binom{n_i}{j} \alpha_i^j \in \zz[G]$ for $j = 1, \ldots, k$). Moreover, $(1+\alpha_i)^{n_i}$ is still a unipotent unit. 
So we let $n = \prod_i n_i$ and then $(1+\alpha)^n$ is a product of unipotent units of $\zz[G]$ determined by nilpotent elements in the Wedderburn components of $\qq[G]$. 

Let $\mathcal{U}(G)$ be the subgroup of $\mathcal{U}(\zz[G])$ generated by unipotent units. Denote $p\mathcal{U}(G)$ the subgroup of $\mathcal{U}(\zz[G])$ generated by {\it primitive unipotent units}, namely, units of the form $1 + \beta$ where $\beta \in \zz[G]$ is nilpotent such that $\beta = \beta e_i$ for some $i \in I$. Then $p\mathcal{U}(G)$ is a normal subgroup of $\mathcal{U}(G)$. The above discussion shows that $\mathcal{U}(G) / p\mathcal{U}(G)$ is generated by torsion images of unipotents. Clearly, if $\zz[G]$ has ND, then $\mathcal{U}(G) / p\mathcal{U}(G)$ is a trivial group. Conversely, if $\mathcal{U}(G) = p\mathcal{U}(G)$ and $\alpha \in \zz[G]$ is nilpotent, then $1 + \alpha = \prod_{i \in I} (1 + \gamma_i)$ where each factor $1 + \gamma_i$ is a product of primitive unipotent units of the form $1 + \beta$ with $\beta \in \zz[G]$ and $\beta e_i = \beta$. Observe that $1 + \alpha = 1 + \sum_{i \in I} \gamma_i$. 
Hence, we have the following statement.

\begin{thm*}
For a finite group $G$, $\zz[G]$ has \textup{ND} if and only if $\mathcal{U}(G) / p\mathcal{U}(G)$ is a trivial group.
\end{thm*}

\begin{prob}
\label{prob10}
Classify finite groups $G$ such that $\mathcal{U}(G) / p\mathcal{U}(G)$ is a finite group.
\end{prob}

Based on the study in this article, the ND property relates to the number of matrix components. In general, we can ask

\begin{prob}
\label{ques2}
What is the connection between the group $\mathcal{U}(G) / p\mathcal{U}(G)$ and matrix components of $\qq[G]$?
\end{prob}

Finally, we denote by OMC the class of  rational group algebras having at most one matrix component. We finish by giving the following Figure~\ref{figure1} which shows the relations between MJD, ND, SN, SSN and OMC. Note that the SSN groups are not a subclass of the ND groups, as shown for example by the group  $C_3 \rtimes C_8$ (see Remark~\ref{remkSN}(ii)). So, in Figure~\ref{figure1}, the diagonal implication from SSN to ND does not hold.

\begin{figure}[h] 
\begin{center}
\begin{tikzpicture}[scale=1]
\scriptsize
\matrix (m) [matrix of math nodes,row sep=12em,column sep=10em,minimum width=5em, line width=1em]{
\textup{\large MJD} &   & \textup{\large SSN} \\
  & \textup{\large OMC} &   \\
\textup{\large ND} & & \textup{\large SN} \\
};

\path[-stealth]
([yshift=2ex]m-1-3.west) edge [line width=1pt,dotted] 
    node [above,yshift=1.0ex] {\large$\zz[A_4]$} 
    node {\large$\times$} 
([yshift=2ex]m-1-1.east)
([yshift=-3pt]m-1-1.east) edge [line width=1pt]
    node {} 
([yshift=-3pt]m-1-3.west)
([xshift=3pt,yshift=-3pt]m-1-1.south) edge [line width=1pt]
    node{} 
([xshift=3pt,yshift=3pt]m-3-1.north)
([xshift=-2ex,yshift=3pt]m-3-1.north) edge [line width=1pt,dotted]
    node {\large$\times$}
    node [left,xshift=-1.0ex] {\large$\zz[A_4]$} 
([xshift=-2ex,yshift=-3pt]m-1-1.south)
([xshift=-3pt,yshift=-3pt]m-1-3.south) edge [line width=1pt]
    node{} 
([xshift=-3pt,yshift=3pt]m-3-3.north)
([xshift=2ex,yshift=3pt]m-3-3.north) edge [line width=1pt,dotted]
    node {\large$\times$}
([xshift=2ex,yshift=-3pt]m-1-3.south)
([xshift=5ex]m-1-1.south) edge [line width=1pt,dotted]
    node [above] {\Large\textup{?}} 
([xshift=-2ex]m-2-2.north)
([xshift=1ex,yshift=1ex]m-2-2.west) edge [line width=1pt,bend left,dotted]
    node [xshift=1ex] {\large$\times$}
    node [left,xshift=0.5ex,yshift=-2.5ex] {\large$\zz[A_4]$}
([xshift=3ex]m-1-1.south)
([xshift=-2ex]m-2-2.south) edge [line width=1pt,bend left] node {} ([xshift=5ex,yshift=-1ex]m-3-1.north)
([xshift=3ex]m-3-1.north) edge [line width=1pt,dotted]
    node [above] {\Large\textup{?}}
([xshift=2ex,yshift=-2ex]m-2-2.west)
([xshift=3ex]m-2-2.south) edge [line width=1pt] node {} ([xshift=-5ex]m-3-3.north)
([xshift=-3ex,yshift=0.5ex]m-3-3.north) edge [line width=1pt,bend right,dotted]
    node [xshift=1ex] {\large$\times$}
([xshift=-1ex,yshift=-1.5ex]m-2-2.east)
([xshift=-5ex]m-1-3.south) edge [line width=1pt,bend right,dotted]
    node [left,xshift=-3pt,yshift=5pt] {\large$\zz[C_3 \rtimes C_8]$} 
    node [xshift=1pt,yshift=3pt] {\large$\times$}
([xshift=2ex]m-2-2.north)
([xshift=-1ex,yshift=2ex]m-2-2.east) edge [line width=1pt,dotted]
    node [xshift=1ex] {\large$\times$}
    node [xshift=5.5ex,yshift=-5.5ex] {\large$\zz[D_8 \Ydown D_8]$} 
([xshift=-2.5ex]m-1-3.south)
([yshift=3pt]m-3-1.east) edge [line width=1pt]
    node{} 
([yshift=3pt]m-3-3.west)
([yshift=-2ex]m-3-3.west) edge [line width=1pt,dotted] 
    node {\large$\times$} 
    node [below,yshift=-1.0ex] {\large$\zz[C_3 \rtimes C_8]$} 
([yshift=-2ex]m-3-1.east)
;            
\end{tikzpicture}
\caption{Relations between MJD, ND, SN, SSN, OMC}
\label{figure1}
\end{center}
\end{figure}
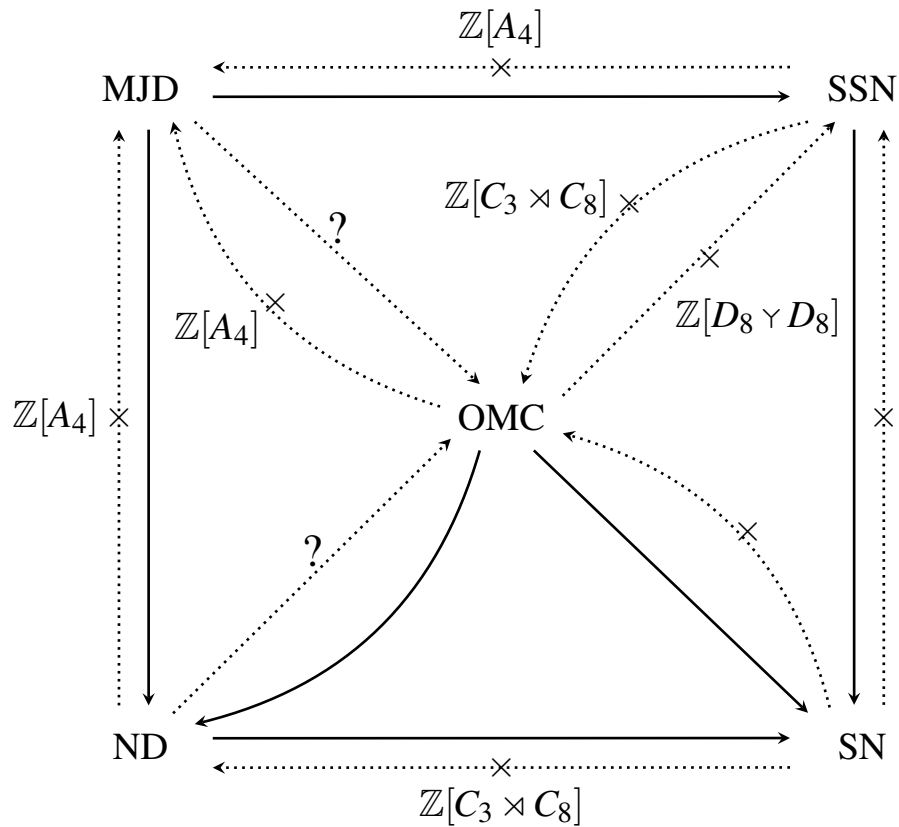

\section*{Acknowledgment}

The first author is supported in part by Onderzoeksraad VUB and FWO (Research Foundation Flanders).
The material in this paper is part of the second author's Ph.D. dissertation written at National Taiwan Normal University. The research of the second author was financially supported by the National Taiwan Normal University (NTNU) within the framework of the Higher Education Sprout Project by the Ministry of Education (MOE) in Taiwan. The second author would like to thank his advisor, Professor Chia-Hsin Liu, for helpful discussions and continuous encouragement. The authors would like to thank the referee for several helpful comments and some suggestions to simplify the proofs of some lemmas.

\bibliographystyle{alphaurl} 
\bibliography{Bib} 

\end{document}